\documentclass[10pt]{amsart} 

\usepackage[usenames,dvipsnames]{xcolor}
\usepackage[bookmarks,colorlinks=true,citecolor=OliveGreen,linkcolor=RoyalBlue]{hyperref}
\usepackage{amssymb,amsthm}
\usepackage[inline]{enumitem}
\usepackage{mathrsfs}
\usepackage{mathtools}
\usepackage{mathabx}
\usepackage{thmtools}
\usepackage[capitalise]{cleveref}		
\usepackage{bm}

\usepackage{lmodern}

\makeatletter
\def\thmt@refnamewithcomma #1#2#3,#4,#5\@nil{%
  \@xa\def\csname\thmt@envname #1utorefname\endcsname{#3}%
  \ifcsname #2refname\endcsname
    \csname #2refname\expandafter\endcsname\expandafter{\thmt@envname}{#3}{#4}%
  \fi
}
\makeatother

\declaretheorem[numberwithin=section]{theorem}
\declaretheorem[sibling=theorem]{proposition}
\declaretheorem[sibling=theorem]{corollary}
\declaretheorem[sibling=theorem]{question}

\declaretheorem[sibling=theorem,style=definition]{definition}

\declaretheorem[sibling=theorem,style=definition]{lemma}

\newcounter{claimCounter}[theorem]

\declaretheorem[sibling=claimCounter,style=remark]{claim}

\newcounter{subClaimCounter}[claimCounter]

\newlist{equivalent}{enumerate}{1}
\setlist[equivalent,1]{label=\textup{(\arabic*)}}

\newlist{sublemma}{enumerate}{1}
\setlist[sublemma,1]{label=\textup{(\alph*)}}

\newlist{sublemma*}{enumerate*}{1}
\setlist[sublemma*,1]{label=\textup{(\alph*)},afterlabel=\hspace{5pt}}

\newlist{orderedlist}{enumerate}{1}
\setlist[orderedlist,1]{label=\textup{(\roman*)}}

\newlist{orderedlist*}{enumerate*}{1}
\setlist[orderedlist*,1]{label=\textup{(\roman*)},afterlabel=\hspace{3pt}}

\newcommand{\seq}[1]{{\left\langle{#1}\right\rangle}}

 \newcommand{\rest}[1]{\restriction{#1}} 
\newcommand{\tth}{{}^{\textup{th}}}
\newcommand{\conc}{\hat{\,\,}}

\newcommand{\andd}{\,\,\,\&\,\,\,}

\newcommand{\Then}{\,\Longrightarrow\,}
\newcommand{\Iff}{\,\,\Longleftrightarrow\,\,}
\renewcommand{\iff}{\leftrightarrow}
\newcommand{\then}{\,\,\rightarrow\,\,}

\DeclareMathOperator{\dom}{dom}
\DeclareMathOperator{\range}{range}

\DeclareMathOperator{\Dek}{Dek}

\newcommand{\converge}{\!\!\downarrow}
\newcommand{\diverge}{\!\!\uparrow}

\newcommand{\w}{\omega}
\newcommand{\s}{\sigma}
\newcommand{\vphi}{\varphi}
\renewcommand{\epsilon}{\varepsilon}

\newcommand{\wock}{\w_1^{\textup{ck}}}

\renewcommand{\le}{\leqslant}
\renewcommand{\ge}{\geqslant}

\renewcommand{\geq}{\geqslant}
\renewcommand{\preceq}{\preccurlyeq}
\renewcommand{\succeq}{\succcurlyeq}
\newcommand{\nle}{\nleqslant}
\newcommand{\nge}{\ngeqslant}

\newcommand{\Tur}{\textup{\scriptsize T}}

\newcommand{\PP}{\mathbb{P}}
\newcommand{\QQ}{\mathbb{Q}}
\newcommand{\MM}{\mathbb{M}}
\newcommand{\GG}{\mathbb{G}}
\newcommand{\SSS}{\mathbb{S}}

\newcommand{\force}{\Vdash}

\newcommand{\UniformIntro}[1]{\le^{\textup{ui}}_{#1}}
\newcommand{\NonUniformIntro}[1]{\le^{\textup{i}}_{#1}}
\newcommand{\nUniformIntro}[1]{\nle^{\textup{ui}}_{#1}}
\newcommand{\nNonUniformIntro}[1]{\nle^{\textup{i}}_{#1}}

\newcommand{\iT}{\NonUniformIntro{\Tur}}
\newcommand{\uiT}{\UniformIntro{\Tur}}
\newcommand{\niT}{\nNonUniformIntro{\Tur}}
\newcommand{\nuiT}{\nUniformIntro{\Tur}}

\newcommand{\iS}{\NonUniformIntro{c.e.}}
\newcommand{\uiS}{\UniformIntro{c.e.}}

\newcommand{\ie}{\NonUniformIntro{e}}
\newcommand{\uie}{\UniformIntro{e}}

\newcommand{\erank}[2]{\lambda(#1;#2)}
\newcommand{\frk}[3]{\lambda^*_{#1}(#2;#3)}
\newcommand{\srk}[2]{\lambda^*(#1;#2)}

\title{Computing sets from all infinite subsets}

\author[N.\:Greenberg]{Noam Greenberg} 
\address{School of Mathematics and Statistics, Victoria University of Wellington, Wellington, New Zealand}
\email{greenberg@msor.vuw.ac.nz}
\urladdr{\url{http://homepages.ecs.vuw.ac.nz/~greenberg/}}

\author[M.\:Harrison-Trainor]{Matthew Harrison-Trainor}
\address{School of Mathematics and Statistics,
Victoria University of Wellington, New Zealand \textit{and} The Institute of Natural and Mathematical Sciences, Massey University, New Zealand}
\email{matthew.harrisontrainor@vuw.ac.nz}
\urladdr{\url{http://homepages.ecs.vuw.ac.nz/~harrism1/}}

\author[L.\:Patey]{Ludovic Patey}
\address{CNRS, Institut Camille Jordan,
	Universit\'e Lyon 1, France}
\email{ludovic.patey@computability.fr}
\urladdr{\url{https://ludovicpatey.com/}}

\author[D.\:Turetsky]{Dan Turetsky}
\address{School of Mathematics and Statistics,
	Victoria University of Wellington, New Zealand}
\email{dan.turetsky@vuw.ac.nz}
\urladdr{\url{http://homepages.ecs.vuw.ac.nz/~dan/}}

\thanks{Greenberg and Turetsky were supported by a Marsden Fund grant \#17-VUW-090.}

\begin{document}

\begin{abstract}
	A set is introreducible if it can be computed by every infinite subset of itself. Such a set can be thought of as coding information very robustly. We investigate introreducible sets and related notions. Our two main results are that the collection of introreducible sets is $\Pi^1_1$-complete, so that there is no simple characterization of the introreducible sets; and that every introenumerable set has an introreducible subset.
\end{abstract}

\maketitle


\section{Introduction} 
\label{sec:introduction}

What information can be coded into all infinite subsets of some set of natural numbers? In one extreme, Soare \cite{Soare:higher} constructed a set which is computable from none of its coinfinite subsets. In the other extreme, an infinite\footnote{The definition is vacuously satisfied when applied to finite sets, so we restrict our attention to infinite sets, sometimes without otherwise mentioning the fact, when talking about introreducible sets and related notions.} set is \emph{introreducible} if it is computable from all of its infinite subsets. In some sense, introreducibility captures a property of having redundant information content. 

Introreducible sets were introduced by Dekker and Myhill \cite{DekkerMyhill58} as a property of retraceable sets, but have perhaps surpassed the latter in importance. The simplest example of introreducible sets are the Dekker sets: given a set $A$ of natural numbers, view $A$ as an infinite binary sequence (with $A(i) = 1 \Longleftrightarrow i \in A$), and form the set $\Dek(A)$ of all initial segments of this binary sequence. From $A$ one can compute $\Dek(A)$ and vice versa, so that $A \equiv_T \Dek(A)$. Moreover, given any infinite subset $B$ of $\Dek(A)$, $B$ may be missing certain initial segments of $A$, but nevertheless it contains arbitrarily long initial segments of $A$. So from $B$ we can still recover $A$ and hence $\Dek(A)$. Thus $\Dek(A)$ is introreducible; and every Turing degree contains an introreducible set.

A second simple example of introreducible sets arises from self-moduli. If $g\colon \w\to \w$ is an increasing self-modulus---meaning that every~$f$ dominating~$g$ computes~$g$---then the range of~$g$ is introreducible.
The simplest example of such a $g$ is the settling time of a computably enumerable set $A$. Such a set $A$ has a computable approximation $(A_s)_{s \in \omega}$ and the settling time function is $g(n) = s$ for the least stage $s$ such that the computable approximation to $A$ has settled on the first $n$ numbers: $A_s \upharpoonright_n = A \upharpoonright_n$. That is, if some $i < n$ is going to enter $A$, it must have done so by stage $g(n)$. Given $g$ we can compute $A$, and vice versa; moreover, given any other function $f$ that dominates $g$, $f$ can compute $A$ (and hence $g$), because to compute $A \upharpoonright_n$ we can run the approximation $A_s \upharpoonright_n$ up to stage $f(n) \geq g(n)$. So $g$ is an increasing self-modulus. Now let $B$ be the range of $g$. Since $g$ is increasing, the $n$th element of $B$ is $g(n)$. $B$ is introreducible because given any infinite subset $C$ of $B$, the $n$th element of $C$ is greater than the $n$th element $g(n)$ of $B$, and so the principal function $f(n)$ of $C$ (where $f(n)$ is the $n$th element of $C$) dominates $g(n)$; this $f$ computes $g$ and $B$.

Perhaps the key question is: What makes a set introreducible? Are all introreducible sets built by some combination of the two methods we just described?

\smallskip{}

Introreducible sets were studied in detail by Jockusch \cite{Jockusch68:Introreducible}, who introduced the notion of \emph{uniformly introreducible sets}: infinite sets which are computable from each of their infinite subsets via a single reduction procedure. Lachlan (see \cite{Jockusch68:Introreducible}) constructed an example of an introreducible set which is not uniformly introreducible. Jockusch showed that in several ways, the uniform notion is more tractable. For example, he showed that if a set and its complement are both uniformly introreducible, then it is computable; in contrast, he was only able to show that if a set and its complement are introreducible, then it is~$\Delta^1_2$. Solovay \cite{Solovay78} improved this bound to~$\Delta^1_1$. It was only much later when Seetapun and Slaman~\cite{SeetapunSlaman95} showed, using Seetapin's results on Ramsey's theorem together with an argument of Jockusch, that if a set and its complement are both introreducible, then it is computable.

Jockusch left open a number of other questions about introreducible sets, but since Seetapun and Slaman's work, the study of introreducible sets has been somewhat dormant. Recently there has been great interest in a number of related concepts which involve computing with all of the subsets of a given set, and so it seems an appropriate time for introreducible sets to make a return. We answer several key questions that were left open for the fifty years since \cite{Jockusch68:Introreducible}.

This paper has two main themes. First, can we give a simple characterization of introreducible sets?  Perhaps we can show that every uniformly introreducible set can be constructed as some combination of initial segments and self-moduli. Jockusch asked whether the collection of (uniformly) introreducible sets is~$\Pi^1_1$-complete, which would give a strong negative answer. In this paper we give an answer to Jockusch's question by showing that the collection of (uniformly) introreducible sets is indeed~$\Pi^1_1$-complete. In doing so, we introduce new ways of constructing introreducible sets which offer much more flexibility than the two methods described above; in essence, our proofs show that in general, one must understand introreducibility via ordinals and paths through trees.

The second line of enquiry considers strengthening introreducibility properties by passing to subsets. For example, we can ask whether every infinite introreducible set has an infinite uniformly introreducible subset. This kind of question arises from Jockusch's investigation of the related notion of \emph{introenumerable} sets, namely infinite sets which are c.e.\ relative to each of their infinite subsets. Every c.e.\ set has a computable subset, and so if~$A$ is uniformly introenumerable, then we can imagine that from the subsets of~$A$ we could somehow compute some fixed subset of~$A$. Thus, Jockusch asked whether every infinite uniformly introenumerable set has an infinite uniformly introreducible subset. We answer this question in the affirmative. 

\medskip{}

Introreducibility is related in spirit to the Ramsey-type problems which have been at the forefront of recent reverse mathematics. Problems such as Ramsey's theorem or the pigeonhole principle have the property that any infinite subset of a solution is itself a solution, and a key step in their analysis has been to look at what can be computed by every solution. (Indeed, introreducibility can be thought of as similar to a one-sided version of the pigeonhole principle.) Indeed as mentioned above Seetapun and Slaman's~\cite{SeetapunSlaman95} that an introreducible and co-introreducible set is computable was from Seetapun's work on Ramsey's theorem, namely that one can find solutions to Ramsey's theorem which avoid computing a particular set. More recently, what can be computed from all subsets of a given set played a key role in Monin and Patey's completion of the last step in the 40-year old program of the computability-theoretic analysis of Ramsey's theorem \cite{monin2019srt22}, namely the separation of $\mathsf{SRT}^{2}_2$ from $\mathsf{RT}^2_2$.

Introreducibility is not the only approach to studying information redundancy. Another recent strand of work has been on coarse and generic computability, e.g.\ \cite{MR2901074,Igusa13,Astor15,HJMS,JockuschSchupp17,Hirschfeldt20}, where for example a set $A$ is generically reducible to a set $B$ if there is a Turing functional $\Phi$ that given any density-1 partial oracle for $B$ (i.e., an oracle that does not always answer, but when it does answer it gives the correct answer), computes $A$ on a subset of density one, and is always correct when it answers. The set~$A$ must be coded redundantly in~$B$, as the computation has to work no matter which parts of the oracle $B$ we lose access to.

Introreducible sets (usually the Dekker set) have also proved useful in constructions, most famously in the work of Slaman and Woodin on definability in the Turing degrees \cite{SlamanWoodin} but also  e.g.\ in \cite{FokinaRosseggerSanMauro}. These are not deep applications of the theory of introreducible sets, but they show that introreducibility is a natural and useful notion.

\smallskip{}

\subsection{Binary relations} 
\label{sub:binary_relations}

Introreducibility was defined as a property of a single set, but our investigations lead us to consider the following notions between two infinite sets.

\begin{definition} \label{def:intro_reductions} 
	Let $B$ be an infinite set.
	\begin{sublemma}
		\item A set~$A$ is \emph{introcomputable} from a set~$B$ if every infinite subset of~$B$ computes~$A$. We write $A \iT B$.

		\item A set~$A$ \emph{uniformly introcomputable} from~$B$ if there is a fixed Turing reduction (functional)~$\Phi$ such that for every infinite $S\subseteq B$ we have $\Phi(S) = A$. We write $A \uiT B$. 
	\end{sublemma}
\end{definition}

\noindent Thus, a set~$A$ is introreducible if it is \emph{self-introcomputable}, that is, if $A\iT A$; similarly, it is uniformly introreducible if $A \uiT A$. In some ways, it appears that these binary relations are more fundamental, certainly more tractable, than the reflexive notions of introreducibility. For example, we will show that uniformity can be achieved by passing to a subset:

\begin{theorem} \label{thm:uniformization_theorem}
	If $A\iT B$, then there is some infinite $C\subseteq B$ such that $A \uiT C$. 
\end{theorem}

\noindent On the other hand, we still do not know whether every introreducible set has a uniformly introreducible subset; the issue is that when $A = B$ is introcomputable, by passing to a set $C \subseteq A = B$, we are changing both the set we are trying to compute and the set doing the computing at the same time.

\medskip

\Cref{def:intro_reductions} can be extended to binary relations other than Turing reducibility. For example, for enumerability, we write $A \iS B$ if $A$ is c.e.\ relative to every infinite subset of~$B$, and $A \uiS B$ if a single enumeration procedure enumerates~$A$ from every infinite subset of~$B$. Considering the question of improving introenumerability to introreducibility, we show:

\bigskip

\begin{theorem} \label{thm:failure_of_improving_enumerability} \
\begin{sublemma}
	\item \label{item:failure_right}
	There are sets~$A$ and~$B$ such that $A \uiS B$, but there is no infinite $C\subseteq B$ such that $A \iT C$. 
	
	\item \label{item:failure_left}
	There are sets~$A$ and~$B$ such that $A \uiS B$, but there is no infinite $C\subseteq A$ such that $C \iT B$. 
\end{sublemma}
\end{theorem}

We were thus surprised to obtain our first main result, namely:

\begin{theorem} \label{thm:improving_enumerability}
	If $A \uiS A$ and $A$ is infinite then there is some infinite $B\subseteq A$ such that $B \uiT B$. That is, every uniformly introenumerable set has a uniformly introreducible subset. 
\end{theorem}

En route to proving this result, we show certain limitations of \cref{thm:failure_of_improving_enumerability}\ref{item:failure_right}:

\begin{theorem} \label{thm:improving_enumerability_one_sided}
	Suppose that $A \iS B$, and that either
	\begin{sublemma}
		\item \label{item:noncollapsing}
		$\w_1^B = \wock$, or
		
		\item  \label{item:introenumerable}
		$A$ is introenumerable. 
	\end{sublemma}
	Then there is some infinite $C\subseteq B$ such that $A \uiT C$. 
\end{theorem}



\smallskip

Finally, we turn to completeness. We show:

\begin{theorem} \label{thm:main_final_completeness_theorem}
	Uniformly, given a linear ordering~$\+L$, we can construct a set~$A$, co-c.e.\ in~$L$, such that:
	\begin{itemize}
		\item If $\+L$ is well-founded, then $A$ is uniformly introreducible; and
		\item If $\+L$ is ill-founded, then~$A$ is not introreducible. 
	\end{itemize}
\end{theorem}

Immediately, this shows that the collection of introreducible sets is ${\bf{\Pi}}^1_1$-complete (for sets of reals) under Borel reductions, witnessed by a Baire class 1 function; the same holds for the collection of uniformly introreducible sets. As mentioned above, this answers one of Jockusch's questions from~\cite{Jockusch68:Introreducible}. Restricted to the computable realm, \cref{thm:main_final_completeness_theorem} shows that the collection of~$\Pi^0_1$ indices for (uniformly) introreducible co-c.e.\ sets is ${{\Pi}}^1_1$-complete, under many-one reducibility. Indeed, using Soare's notation (see \cite{Soare1987Recursively}), it shows that the pair $(N,U)$ is $(\Sigma^1_1,\Pi^1_1)$-$m$-complete, where $N$ is the collection of~$\Pi^0_1$ indices of co-c.e.\ sets which are not introreducible, and~$U$ is the set of~$\Pi^0_1$ indices of co-c.e.\ uniformly introreducible sets. Informally, what this result means is that there is no characterization of the introreducible or uniformly introreducible sets that is simpler than the naive definition: the simplest way to decide whether a set is introreducible is to check whether it is computable from all of its infinite subsets.

\medskip

The paper is organised as follows. In \cref{sec:basics}, we fix notation, establish basic properties of the relations we investigate, and prove some lemmas that are useful later on. The rest of the paper is organised by technique. In \cref{sec:forcing_methods} we use Cohen and Mathias forcing, to prove, for example, \cref{thm:uniformization_theorem} and related results, as well as answer another question of Jockusch's: we show that an introreducible set cannot be co-hyperimmune. In \cref{sec:completeness} we prove \cref{thm:main_final_completeness_theorem}. This proof involves a technique for building introreducible sets which is also used in the proof of \cref{thm:failure_of_improving_enumerability}\ref{item:failure_left}; we therefore provide that proof in the same section. Finally, in \cref{sec:improving_enumerability} we prove \cref{thm:improving_enumerability_one_sided,thm:improving_enumerability}. The proof of \cref{thm:failure_of_improving_enumerability}\ref{item:failure_right} is relatively short but is unrelated to other proofs, and so we place it in \cref{sec:basics}.

\smallskip

As mentioned above, we regard the following as the main question which is left open:

\begin{question} \label{qn:uniorm_subset}
Does every infinite introreducible set have an infinite uniformly introreducible subset? Does every infinite introenumerable set have an infinite uniformly introenumerable subset?
\end{question}

Related is the question of whether every introenumerable set has an introreducible subset. Of course a positive answer to the second part of \cref{qn:uniorm_subset} implies a positive answer to this question. 





\section{Basics} 
\label{sec:basics}

\subsection{Notation} 
\label{sub:notation}

Our notation is fairly standard. Following Ramsey theory (and set theory in general), for a set~$A$, we let $[A]^\w$ denote the collection of all infinite subsets of~$A$, while we let $[A]^{<\w}$ denote the collection of all finite subsets of~$A$. For the most part, unless otherwise noted most of the sets that appear are infinite. For a Turing functional~$\Phi$ and a set~$X$, we write $\Phi(X)$ to denote the set computed from~$X$ using~$\Phi$; we write $\Phi(X,n)$ to denote the output of the procedure~$\Phi$ on input~$n$ with oracle~$X$. If $f,g\in \w^\w$ then we write $f\ge g$ if $f(n)\ge g(n)$ for all~$n$. If $f\in \w^\w$ and $\s\in \w^{<\w}$ then we write $\s\le f$ if $\s(n)\le f(n)$ for all $n< |\s|$. 
For any infinite set $A$, $p_A$ is the principal function of~$A$: the increasing enumeration of~$A$.

\subsection{Enumeration reducibility} 
\label{sub:enumeration}

For any binary relation~$r$ on $[\w]^\w$, we write $A \NonUniformIntro{r} B$ if $A r C$ holds for every $C\in [B]^\w$. If there is a countable family~$\+C$ of partial functions which determines~$r$ --- meaning that $A r B$ if and only if there is some $\Phi\in \+C$ such that $\Phi(B) = A$ --- then we write $A \UniformIntro{r} B$ if there is some $\Phi\in \+C$ such that $A = \Phi(C)$ for all $C\in [B]^\w$. We have already seen this notation with~$r$ being either Turing reducibility, or the relation ``c.e.\ in''. We apply it to one more reducibility, namely, enumeration reducibility~\cite{FriedbergRogers}. Recall that $A \le_e B$ if there is a procedure which outputs positive information about~$A$ using positive information about~$B$. Formally, if there is a c.e.\ collection~$\Psi$ of pairs of finite sets such that $A = \Psi(B) = \bigcup \left\{ F \,:\,  (\exists E\subseteq B)\,\,(E,F)\in \Psi  \right\}$. We call~$\Psi$ an \emph{enumeration reduction} or an \emph{enumeration functional}.  Then $A \ie B$ if~$A$ is enumeration reducible to every $C\in [B]^\w$, and $A \uie B$ if there is an enumeration reduction~$\Psi$ such that $A = \Psi(C)$ for all $C\in [B]^\w$. 

The interest in these relations stems from the following:

\begin{proposition} \label{prop:uie_and_uiS}
	For infinite sets~$A$ and~$B$, we have $A \uie B$ if and only if $A \uiS B$. 
\end{proposition}

To avoid confusion, we will call procedures which enumerate a set using both positive and negative information from the oracle \emph{relative c.e.\ operators}.\footnote{Formally, these are c.e.\ sets of pairs $(\s,x)\in 2^{<\w}\times \w$; for such a set~$\Theta$ and $\tau\in 2^{\le \w}$, we let  $\Theta(\tau)$ be the set of $x\in \w$ for which there is some $\s\preceq \tau$ such that $(\s,x)\in \Theta$.}

\begin{proof}
	One direction is immediate: every enumeration reduction can be turned into a relative c.e.\ operator (which ignores negative information). Suppose that $A \uiS B$ via a relative c.e.\ operator~$\Theta$. Define an enumeration reduction~$\Psi$ by \emph{forgetting} the negative information: for a finite set~$F$, enumerate a number~$n$ in $\Psi(F)$ if there is some axiom in~$\Theta$ which enumerates~$n$ using an initial segment~$\s$ of an oracle, and~$F$ is the finite set $\left\{ k \,:\,  \s(k) = 1 \right\}$. To see that this works, it is clear that for any subset $C\in [B]^\w$ we have $A\subseteq \Psi(C)$. On the other hand, suppose that $n\in \Psi(C)$, given by some $F\subseteq C$ which comes from an axiom $(\s,n)$ in~$\Theta$. It is possible that $\s$ is not an initial segment of~$C$, but it is an initial segment of an infinite subset~$D$ of~$C$, and $A = \Theta(D)$, whence $n\in A$. 
\end{proof}

The transitivity of $\le_e$, in contrast with the relation ``c.e.\ in'', means that it is often easier to work with the relation $\uie$ rather than~$\uiS$. Note that \cref{prop:uie_and_uiS} implies that $A\uiT B$ if and only if there is a functional which computes~$A$ from all of $B$'s infinite subsets using only positive information; this is because $A \uiT B$ if and only if $A\oplus \overline{A} \uiS B$.

\subsection{Subsets of hyperarithmetic sets} 
\label{sub:subsets_of_hyperarithmetic_sets}

The following is the binary-relation version of Jockusch's result (from \cite{Jockusch68:Introreducible}) that every hyperarithmetic set has a uniformly introreducible subset.

\begin{lemma} \label{lem:3.9}
	Suppose that $A \uiT B$ and $B\in \Delta^1_1(A)$. Then there is some $C\in [B]^\w$ such that $C \uiT C$. 
\end{lemma}

The proof gives $C \in \Delta^1_1(A)$.

\begin{proof}
	Let $\alpha < \w_1^A$ be an $A$-computable ordinal such that $B \le_\Tur A^{(\alpha)}$. The set $A^{(\alpha)}$ has a uniform self-modulus relative to~$A$: there is a function $f\equiv_\Tur A^{(\alpha)}$ and a functional~$\Phi$ such that for every $g\ge f$, $\Phi(A,g)= A^{(\alpha)}$. Let $C\subseteq B$ be the subset whose principal function is $p_B\circ f$, that is, the collection of $f(n)\tth$ elements of~$B$ (for $n\in \w$). Then $C\le_\Tur A^{(\alpha)}$. Any infinite $D\subseteq C$ can uniformly compute~$A$, and its principal function majorizes~$f$, and so using~$\Phi$, $D$ can (uniformly) compute $A^{(\alpha)}$, and hence~$C$. 
\end{proof}

\subsection{Proof of \cref{thm:failure_of_improving_enumerability}\ref{item:failure_right}} 
\label{sub:failure_right}

Toward proving \cref{thm:failure_of_improving_enumerability}\ref{item:failure_right}, we first observe a positive-sided version of Solovay's result \cite{Solovay78} that the hyperarithmetic sets are those sets with uniform moduli. The following proposition says that a set is $\Pi^1_1$ if and only if it has a \emph{uniform c.e.\ modulus} (which is a function $f$ as in (1) below).

\begin{proposition} \label{prop:c.e._moduli}
	The following are equivalent for a set $A\subseteq \w$:
	\begin{equivalent}
		\item There is a function $f\in \w^\w$ and a relative c.e.\ operator~$\Theta$ such that for all $g\ge f$, $\Theta(g) = A$. 
		\item $A$ is $\Pi^1_1$. 
	\end{equivalent}
\end{proposition}

\begin{proof}
	$(2)\Rightarrow(1)$: As $A$ is $\Pi^1_1$, there is a uniformly computable sequence $(T_n)_{n \in \omega}$ of trees in Baire space such that $T_n$ is well-founded if and only if $n \in A$.  For each $n \notin A$, fix a path $f_n \in [T_n]$.  Let
\[
f(x) = \sum_{\substack{n \not \in A\\n \le x}} f_n(x-n).
\]
We claim that $f$ is a uniform c.e.\ modulus for~$A$, i.e., as required for~(1). 

Given $g \ge f$, let $\mathcal{C}_n = \{ h \in \omega^\omega \,:\, \forall x\, [h(x) \le g(x+n)]\}$.  Then for each ill-founded~$T_n$, $f_n(x) \le f(x+n) \le g(x+n)$, so $f_n \in \mathcal{C}_n$.  Further, $\mathcal{C}_n$ is effectively compact relative to~$g$.  So $n \in A$ if and only if $[T_n] = \emptyset$ if and only if $[T_n]\cap \mathcal{C}_n = \emptyset$, and the latter is a $\Sigma^0_1(g)$ question.  Finally, this process is uniform, so~$f$ is a uniform c.e.\ modulus for $A$.

\smallskip

$(1)\Leftarrow(2)$: Suppose there is some~$f$ and  a relative c.e.\ operator~$\Theta$ such that $A = \Theta(g)$ for all $g \ge f$.  Then the following is a $\Pi^1_1$ description of $A$:
\[
x \in A \iff (\forall h \in \omega^\omega) (\exists \sigma \in \omega^{<\omega}) [ (\sigma \ge h) \andd (x \in \Theta(\sigma))].
\]

If $x \in A$, then for any~$h$, fix a~$g$ with $g \ge h$ and $g \ge f$.  Then $A = \Theta(g)$ by assumption, and so there is some $\sigma \prec g$ with $x \in \Theta(\sigma)$, and thus $x$ satisfies the righthand side.

If $x$ satisfies the righthand side, then in particular it satisfies it for $h = f$.  Pick a witnessing~$\sigma$, and extend it to a function $g \ge f$.  Then $x \in \Theta(\sigma) \subseteq \Theta(g) = A$.
\end{proof}

The proof of \cref{thm:failure_of_improving_enumerability}\ref{item:failure_right} relies on the notion of {computable encodability} investigated by Solovay \cite{Solovay78}. A set~$A$ is \emph{computably encodable} if every infinite set has a subset which computes~$A$. Every hyperarithmetic set~$A$ has a modulus, which implies that it is computably encodable (given any set, thin it to a subset sufficiently sparse so that its principal function majorizes the modulus of~$A$). Solovay showed that the computably encodable sets are precisely the hyperarithmetic ones. 

\begin{proof}[Proof of \cref{thm:failure_of_improving_enumerability}\ref{item:failure_right}]
	Fix any $\Pi^1_1$ set $A$ which is not~$\Delta^1_1$ (for example Kleene's~$\+O$). Then by \cref{prop:c.e._moduli}, there is a function $f \in \w^\w$ and a relative c.e.\ operator~$\Theta$ such that $A = \Theta(g)$ for every $g \ge f$. Since $A \not \in \Delta^1_1$, as mentioned, it is not computably encodable: there is a set $D \in [\omega]^\omega$ such that for every $B \in [D]^\omega$, $B \nge_\Tur A$. Let $B \in [D]^\omega$ be such that $p_B \ge f$. Then $A \uiS B$, but $A\nle_\Tur B$, so certainly $A \niT B$. 
\end{proof}

Note that in light of \cref{thm:improving_enumerability_one_sided}\ref{item:noncollapsing}, it is not surprising that the sets~$A$ and~$B$ produced are quite complicated.


\section{Forcing methods} 
\label{sec:forcing_methods}

\subsection{Cohen subsets of an infinite set} 
\label{sub:cohen_subsets_of_an_infinite_set}

Let~$A$ be an infinite set. We let $\PP_A$ be the collection of all Cohen conditions $\s\in 2^{<\w}$ which are characteristic functions of finite subsets of~$A$, ordered by extension. For our notation, it will be convenient to identify finite nonempty sets with finite binary strings as follows: a set~$F$ is identified with its characteristic function of length $\max F + 1$. Thus, for finite sets~$E$ and~$F$, we write $E\preceq F$ if~$F$ is an \emph{end-extension} of~$E$: $E\subseteq F$ and $\min (F\setminus E) > \max E$. The collection of all Cohen conditions which correspond to finite subsets of~$A$ is dense in~$\PP_A$. 

It is clear that since~$A$ is infinite, a sufficiently generic filter gives (a characteristic function of) an infinite subset~$G$ of~$A$. 

\subsubsection{Transitivity of introreduction relations} 
It is clear that the relations $\iT$ and $\uiT$ are transitive; indeed if $A \le_\Tur B$ and $B \iT C$ then $A \iT C$, and the same holds for $\uiT$. \Cref{prop:uie_and_uiS} implies that the relation~$\uiS$ is transitive as well. The relation ``c.e.\ in'' is not transitive; nonetheless, we can prove:

 \begin{theorem} \label{thm:transitivity_for_ie}
 	The relation $\iS$ is transitive. 
 \end{theorem}
 
\begin{proof}
	Suppose that $A \iS B \iS C$. Since every generic $G\subset B$ enumerates~$A$, there is a condition $\s\in \PP_B$ and a relative c.e.\ operator~$\Theta$ such that $\s \force\Theta(G)=A$. This means:
	\begin{itemize}
		\item For every $n\in A$ there is some $\tau\in \PP_B$ extending~$\s$ such that $n\in \Theta(\tau)$; and
		\item For every~$\tau$ extending~$\s$, $\Theta(\tau)\subseteq A$. 
	\end{itemize}
	Let $X$ be a subset of~$C$. Then $B$ is c.e.\ in~$X$. Of course this means that the collection of finite subsets of~$B$ is $X$-c.e. We can then enumerate~$A$ using~$X$ as follows: 
		\[
			A = \bigcup \left\{ \Theta(F) \,:\,  F\in [B]^{<\w} \andd \s\preceq F \right\}. \qedhere
		\]
\end{proof}

The following generalises \cite[Theorem 5.3]{Jockusch68:Introreducible} (which also uses a finite extension argument):

\begin{corollary} \label{thm:transitivity_for_iT_and_ie}
	If $A \iT B$ and $B \iS C$ then $A \iT C$. 
\end{corollary}

\begin{proof}
	As mentioned above, $A \iT X$ if and only if $A\oplus \overline{A} \iS X$.
\end{proof}


Suppose that $B \iS C$  but that $A \niT C$. The proof of \cref{thm:transitivity_for_ie} shows that for no Turing functional~$\Phi$ is there a condition $\s\in \PP_B$ which forces that $\Phi(G)=A$. Computably in~$A\oplus B'$ we can then construct an infinite $G\subset B$ which for every~$\Phi$, either some $\s\prec G$ forces that $\Phi(G)$ is partial, or forces some disagreement between $\Phi(G)$ and~$A$. Thus, if $B \iS C$ but $A \niT C$ then there is some $G \subseteq B$ infinite with $G  \le_\Tur A\oplus B'$ such that $A \nle_\Tur G$. 
By taking $A = B = C$, we obtain:

\begin{theorem} 
	If $A$ is introenumerable, then $A$ is introreducible if and only if $A$ is computable from all of its infinite subsets which are $\Delta^0_2(A)$. 
\end{theorem}

A quick examination of the proof of \cref{thm:transitivity_for_ie} also shows that we can add uniformity: if $A \iS B$ and $B \uiS C$ then $A \uiS C$.

\subsubsection{Hyperimmunity} 

Jockusch asks at the end of \cite{Jockusch68:Introreducible} whether an introreducible set can have a hyperimmune complement. We show it cannot. The following argument is a straightforward modification of \cite[Proposition 4.4]{Hirschfeldt2008strength}.

\begin{proposition} \label{prop:hyperimmune_avoidance}
	Suppose that~$C$ is noncomputable and that $A$ is co-hyperimmune. Then there is an infinite subset $G\subseteq A$ which does not compute~$C$. 
\end{proposition}

\begin{proof}
	We will show that a~$G$ sufficiently generic for~$\PP_A$ does not compute~$C$. Suppose, for a contradiction, that this is not the case. Then there is a functional~$\Phi$ and a finite set $E\subset A$ which (thought of as an element of $\PP_A$) forces that $\Phi(G)=C$. First, we observe that for every~$k$ there is a pair $(F_0,F_1)$ of finite sets satisfying:
	\begin{orderedlist*}
		\item For some~$n$ we have $\Phi(E\cup F_0,n)\converge = 0$ and $\Phi(E\cup F_1,n)\converge = 1$; and
		\item $\min (F_0\cup F_1) > k, \max E$.
	\end{orderedlist*}
	Otherwise, if $k > \max E$ is sufficiently large so that no such pair $(F_0,F_1)$ exists, then we can compute~$C$ by outputting $\Phi(E\cup F,n)$ for finite sets $F$ with $\min F > k$. Now, since $\overline A$ is hyperimmune, we see that there is some pair $(F_0,F_1)$ satisfying (i) with $F_0\cup F_1 \subset A$. Thus, both $E\cup F_0$ and $E\cup F_1$ are conditions in~$\PP_A$, both extending~$E$.  But then one of $E\cup F_0$ and $E\cup F_1$ forces that $\Phi(G)\ne C$, a contradiction. 
\end{proof}

By letting $A = C$ (and noting that hyperimmune sets cannot be computable), we obtain:

\begin{corollary} \label{cor:hyperimmune}
	An introreducible set cannot be co-hyperimmune. 
\end{corollary}

\subsection{Global Mathias forcing} 
\label{sub:global_mathias_forcing}

The conditions of Mathias forcing are pairs $(F,X)$ where $F$ is finite, $X$ is infinite, and $\max F < \min X$. A condition $(F,X)$ extends a condition $(E,Y)$ if $X\subseteq Y$ and $E \subseteq F \subset E\cup Y$. In computability, the reservoirs~$X$ are often restricted to some countable collection of sets, such as the low sets, or sets in a Turing ideal. Here, however, we use the unrestricted version, allowing all possible reservoirs. 

Recall that a condition $(F,X)$ forces some statement $\vphi(G)$ if the statement holds for every \emph{sufficiently generic} $G$ compatible with $(F,X)$ (where $G$ is compatible with $(F,X)$ if $F\subseteq G \subseteq F\cup X$). We say that a condition $(F,X)$ \emph{strongly forces} $\vphi(G)$ if the statement holds for \emph{every} infinite set $G$ compatible with $(F,X)$. 

The main combinatorial tool used is the \emph{Galvin-Prikry theorem} \cite{GalvinPrikry} which states that Borel subsets of $[\w]^\w$ are \emph{Ramsey}, namely: they or their complements contain $[A]^\w$ for some infinite set~$A$. This gives the strong Prikry property of Mathias forcing:

\begin{proposition} \label{prop:Prikry}
	If $(F,X)$ is a Mathias condition and $\+U\subseteq [\w]^\w$ is Borel, then there is an infinite set $Y\subseteq X$ such that the condition $(F,Y)$ strongly decides the statement $G\in \+U$. 
\end{proposition}

\begin{proof}
	The set $\left\{ Z\in [X]^\w \,:\,  F\cup Z \in \+U    \right\}$ has the Ramsey property (relative to the space $[X]^\w$). 
\end{proof}

\subsubsection{Uniformization} 

Mathias forcing helps uniformise introcomputations and enumerations. 

\begin{proposition} \label{prop:uniformization_theorem_for_c.e.}
	If $A\iS B$ then there is some $C\in [B]^\w$ such that $A \uiS C$. 
\end{proposition}

Note that again by considering $A\oplus \overline{A}$, \cref{thm:uniformization_theorem} follows. 

\begin{proof}
	The Mathias condition $(\emptyset, B)$ strongly forces that $A$ is c.e.~in $G$, where~$G$ denotes the generic. Thus, there is some relative c.e.\ operator~$\Theta$ and a condition $(F,X)$ extending $(\emptyset, B)$ (so $X\subseteq B$) which forces that $\Theta(G)=A$. By the strong Prikry property (\cref{prop:Prikry}), by shrinking~$X$ we may assume that $(F,X)$ strongly forces that $\Theta(G)=A$. Then $A \uiS X$: the map $Z\mapsto \Theta(F\cup Z)$ sends all infinite subsets of~$X$ to~$A$.
\end{proof}


We note that there is nothing special about Turing, indeed the proof applies to any reducibility defined by a countable collection of partial Borel functions. 

\subsubsection{Sets without introreducible subsets} 

Much of our focus is on finding introreducible subsets, but Mathias forcing also allows us to construct sets without introreducible subsets. 

\begin{proposition} \label{prop:no_subsets}
	Every infinite set has an infinite subset~$B$ satisfying: for all $C,D\in [B]^\w$, $C \nuiT D$. 
\end{proposition}

\begin{proof}
	We claim that a sufficiently Mathias generic~$G$ has the desired property; it is important to note that meeting only countably many dense sets suffices. (To get~$G$ to be a subset of a given set~$A$, we can start with the condition $(\emptyset, A)$.) Let~$\Phi$ be a Turing functional. Given any condition $(F,X)$, we will find an extension which strongly forces that for all $C,D\in [G]^\w$, there is some $Z\in [D]^\w$ such that $\Phi(Z)\ne C$, and so that~$\Phi$ does not witness that $C \uiT D$. 

	For each $n> \max F$ let $\+U_n = \left\{  Z\in [X]^\w \,:\,  \Phi(Z,n)\converge = 0  \right\}$. There are two cases. First, suppose that for some $n> \max F$ there is some $Y\in [X]^\w$ such that $[Y]^\w\cap \+U_n = \emptyset$. Then $(F,Y\setminus \{n\})$ is the desired extension, as in fact for all~$G$ compatible with $(F,Y\setminus \{n\})$, for all $C,D\in [G]^\w$ we have $\Phi(D)\ne C$. Otherwise, we claim that $(F,X)$ itself is as required: suppose that~$G$ is compatible with $(F,X)$, and let $C,D\in [G]^\w$. Let $n\in C$ be greater than~$\max F$. Then $D\setminus F$, which is an infinite subset of~$X$, has an infinite subset~$Z$ in $\+U_n$, and so $\Phi(Z)\ne C$. 
\end{proof}

\begin{corollary} \label{cor:no_subset}
	Every infinite set has an infinite subset~$B$ satisfying: for all $C,D\in [B]^\w$, $C \niT D$. In particular, every infinite set has an infinite subset~$B$ which has no introreducible subset. 
\end{corollary}

\begin{proof}
	\Cref{thm:uniformization_theorem} implies that the set~$B$ given by \cref{prop:no_subsets} is as required. 
\end{proof}

Soare \cite{Soare:higher} constructed a set without a subset of strictly higher degree. In fact, his construction gives an infinite set~$B$ satisfying: for all $C,D\in [B]^\w$, if $C \le_\Tur D$ then $C \subseteq^* D$. This implies \cref{cor:no_subset}. We added a proof, since it is simpler than Soare's; indeed, the proof of \cref{prop:no_subsets} does not even use the Galvin-Prikry theorem.

\subsection{Mathias and Spector-Gandy} 
\label{sub:mathias_and_spector_gandy}

We describe a variant of \emph{restricted} (countable) Mathias forcing which is also related to Spector-Gandy forcing (forcing with nonempty $\Sigma^1_1$ classes). The idea of modifying Mathias forcing to allow a collection of possible reservoirs was used by Liu in his separation of Ramsey's theorem for pairs from weak K\"onig's lemma \cite{Liu2012RT22}, and later also in \cite{Liu2015Cone,Patey2017Controlling}. 
Let~$\SSS$ be the following notion of forcing:
\begin{itemize}
	\item Conditions are pairs $(F,\+C)$, where $F$ is finite and $\+C\subseteq [(\max F,\infty)]^\w$ is nonempty, $\Sigma^1_1$, and closed downwards under~$\subseteq$: for all $X\in \+C$, $[X]^\w\subseteq \+C$. 
	\item A condition $(E,\+D)$ extends a condition $(F,\+C)$ if~$E$ extends~$F$, and for all $X\in \+D$, $X\cup (E\setminus F)\in \+C$.
\end{itemize}
If $(E,\+D)\le_{\SSS} (F,\+C)$ then $\+D \subseteq \+C$. Note that $(E,\+D)\le_{\SSS} (F,\+C)$ if and only if~$E$ extends~$F$ and for all $X\in \+D$ there is some $Y\in \+C$ such that $(E,X) \le_{\MM} (F,Y)$, where $\MM$ is Mathias forcing. 

\smallskip

If $\GG\subset \SSS$ is a sufficiently generic filter, then we let 
\[
	G[\GG] = \bigcup \left\{ F \,:\,  (\exists \+C)\,\, (F,\+C)\in \GG \right\}. 
\]
It is not difficult to see that~$G$ is infinite: for any condition $(F,\+C)$ we can take any $n> \max F$ such that $n\in X$ for some $X\in \+C$; then $(F\cup\{n\}, \+D)$ is an extension of $(F,\+C)$, where $\+D = \left\{ Y \subseteq (n,\infty) \,:\, Y\cup \{n\}\in \+C  \right\}$. 

We say that a set~$Z$ is \emph{compatible} with a condition $(F,\+C)$ if~$Z$ extends~$F$ and $Z\setminus F\in \+C$. We say that a condition $(F,\+C)$ \emph{strongly forces} a statement $\vphi(G)$ if $\vphi(Z)$ holds for every infinite set~$Z$ which is compatible with it. The following lemma says that if a condition strongly forces a statement then it forces that statement. 

\begin{lemma} \label{lem:Spector-Gandy}
	Every condition forces that~$G$ is compatible with it. 
\end{lemma}

The argument is similar to the argument that a generic filter for Spector-Gandy forcing determines a generic real, that is, that the intersection of all the~$\Sigma^1_1$ sets in the filter is nonempty. We give the proof for completeness of presentation.

\begin{proof}
	For any class $\+A\subseteq [\w]^\w$ let $\+A^{\supseteq} = \bigcup_{X\in \+A}[X]^\w$ be the downward closure of~$\+A$ in $[\w]^\w$. If~$\+A$ is $\Sigma^1_1$ then so is~$\+A^{\supseteq}$. 

	Let $(F,\+C)$ be a condition; let $T$ be a computable tree defining a closed subset $[T]\subseteq 2^\w\times \w^\w$ such that $\+C = p[T]$. For $(\s,\tau)\in T$ let $T_{\s,\tau}$ be the subtree of~$T$ consisting of all pairs compatible with~$(\s,\tau)$, and let $\+C_{\s,\tau} = p[T_{\s,\tau}]$. 

	For a finite set~$E$ extending~$F$, let $\s_E$ be the string (of length $1+ \max E$) corresponding to the set $E\setminus F$. We show:
	\begin{description}
		\item[($*$)] The collection of conditions $(E,\+D) \le_{\SSS} (F,\+C)$ for which there is some~$\tau$ such that $(\s_E,\tau)\in T$ and $(E,\+D) \le_{\SSS} (F, ({\+C_{\s_E,\tau}})^\supseteq)$ 
		 is dense below $(F,\+C)$. 
	\end{description}

	To see this, let $(E,\+D)$ extend~$(F,\+C)$. Extension in~$\SSS$ implies that 
			$\+D \subseteq \left( \+C \cap [\s_E]^\preceq \right)^\supseteq$,
		where $[\s_E]^\preceq$ is the clopen subset of Cantor space determined by the string~$\s_E$. Now $\+C\cap [\s_E]^\preceq = \bigcup \{ \+C_{\s_E,\tau} \,:\, (\s_E,\tau)\in T\}$, and so
	\[
		\+D \subseteq \left( \bigcup\{ \+C_{\s_E,\tau} \,:\, (\s_E,\tau)\in T\} \right)^\supseteq = 
		\bigcup\{ ({\+C_{\s_E,\tau}})^\supseteq \,:\, (\s_E,\tau)\in T\},
	\]
	 so there is some~$\tau$ such that $(\s_E,\tau)\in T$ and $\+D\cap ({\+C_{\s_E,\tau}})^\supseteq$ is nonempty; it follows that the condition $(E,\+D\cap  ({\+C_{\s_E,\tau}})^\supseteq)$ extends both $(E,\+D)$ and $(F, ({\+C_{\s_E,\tau}})^\supseteq)$.

	 \medskip
	 
	 To prove the lemma, we need to show that if~$\GG$ is sufficiently generic then $G[\GG]$ is compatible with every condition in~$\GG$. Let $(F,\+C)$ be a condition, and let~$\GG$ be a sufficiently generic filter containing $(F,\+C)$; let $G = G[\GG]$. Let~$T$ be as above. By recursion, we build a strictly increasing sequence $(\s_0,\tau_0) \prec (\s_1,\tau_1) \prec \cdots$ of pairs in~$T$ such that $\bigcup_i \s_i = G\setminus F$ and for all~$i$, $(F, (\+C_{\s_i,\tau_i})^\supseteq) \in \GG$. We start with $\s_0 = \tau_0$ being the empty string. Suppose that we have already chosen~$(\s_i,\tau_i)$ as required. By ($*$) applied to the condition $(F, (\+C_{\s_i,\tau_i})^\supseteq)$ (and the tree $T_{\s_i,\tau_i}$), since~$\GG$ is sufficiently generic, we find some $(E,\+D)\in \GG$ with $\max E > |\s_i|$ extending $(F, (\+C_{\s_E,\tau})^\supseteq)$ for some~$\tau$ such that $(\s_E,\tau)\in T_{\s_i,\tau_i}$. Thus, $\tau\succ \tau_i$, so we choose $(\s_{i+1},\tau_{i+1}) = (\s_E,\tau)$, noting that $\s_E \prec G\setminus F$ as $(E,\+D)\in \GG$. Letting $f = \bigcup_i \tau_i$, we see that $(G\setminus F,f)\in [T]$, whence $G\setminus F \in \+C$, as required. 
\end{proof}

Suppose that $A \iS B$. The collection of sets $C\in [B]^\w$ witnessing \cref{prop:uniformization_theorem_for_c.e.} is $\Pi^1_1(A)$, and since such a set was obtained by forcing with unrestricted conditions, the argument does not give any reasonable bound on the complexity of such~$C$. Below, we will apply a basis theorem to the following stronger result. 

\begin{theorem} \label{thm:a_lot_of_uniformisation}
	If $A\iS B$ then there is a nonempty $\Sigma^1_1(B)$ class of sets $C\in [B]^\w$ satisfying $A \uiS C$. 
\end{theorem}

\begin{proof}
	We use $\SSS^B$, the notion of forcing discussed above relativised to~$B$: the conditions $(F,\+C)$ allow~$\+C$ to be $\Sigma^1_1(B)$ rather than merely~$\Sigma^1_1$. 

	The condition $(\emptyset,[B]^\w)$ forces that $G\subseteq B$, and so forces that $A$ is c.e.\ in~$G$. Therefore there is a condition $(F,\+C)$ extending $(\emptyset,[B]^\w)$ which for some relative c.e.\ operator~$\Phi$, forces that $\Phi(G)=A$. 

	For every $n$, let
	\[
		\+U_n = \left\{ Z\in 2^\w \,:\,  n\in A \Iff n\in \Phi(Z) \right\}.
	\]

	\begin{claim} \label{clm:a_lot_of_unif}
		For every~$n$, there is no $Y\in \+C$ such that the condition $(F,Y)$ strongly forces (in Mathias forcing) that $G\notin \+U_n$.
	\end{claim}

	\begin{proof}
	Suppose, for a contradiction, that the claim fails for some~$n$ and~$Y$. 

	There are two cases. First, suppose that $n\notin A$. Then $(F,Y)$ strongly forces (again in Mathias forcing) that $n\in \Phi(G)$. Let~$E$ be an initial segment of~$Y$ extending~$F$ such that $n\in \Phi(E)$. Then 
	 \[
	 (E, \left\{ Z\subseteq (\max E,\infty) \,:\,  Z\cup (E\setminus F)\in \+C   \right\})
	 \]
	 is a condition in~$\SSS^B$ which extends $(F,\+C)$ and strongly forces (in $\SSS^B$) that $\Phi(G)\ne A$; it therefore forces $\Phi(G)\ne A$, contradicting the assumption that $(F,\+C)$ forces the opposite. 

	 \smallskip
	 
	 Next, suppose that $n\in A$. So $(F,Y)$ strongly forces that $n\notin \Phi(G)$. Then 
	 \[
		\+D = 
		 \left\{ Z\in \+C \,:\, (\forall E\in [Z]^{<\w})\,\, n\notin \Phi(F\cup E)       \right\}
	\]
	is $\Sigma^1_1(B)$ and so $(F,\+D)$ is a condition in $\SSS^B$ extending $(F,\+C)$ and strongly forcing that $\Phi(G)\ne A$, which again is impossible. 		
	\end{proof}
	
	By the strong Prikry property of Mathias forcing, for every $X\in \+C$ and each~$n$ there is some $Y\in [X]^\w$ such that $(F,Y)$ strongly forces (with respect to Mathias forcing) that $G\in \+U_n$. Iterating, starting with any $X\in \+C$ obtain a sequence $X \supset Y_0 \supset Y_1 \supset Y_2 \supset \cdots$ with $(F,Y_n)$ strongly forcing that $G\in \+U_n$, and also ensure that $\min Y_{n+1} > \min Y_n$. Let $\tilde X = \left\{ \min Y_n \,:\,  n\in \w \right\}$; then $\tilde X \in [X]^\w$ and is an element of the class~$\+E$ of sets $Z\in \+C$ satisfying: for all~$n$, for every finite $E\subset Z$ which omits the first~$n$ elements of~$Z$, if $n\in \Phi(F\cup E)$ then $n\in A$. This class is $\Sigma^1_1(B)$ ($A$ is c.e.\ in~$B$, so is $\Delta^1_1(B)$). 

	We claim that for every $C\in \+E$ we have $A \uiS C$. Indeed, $\+E$ is downward closed, and we exhibit a uniform procedure for enumerating~$A$ from any $C\in \+E$. For all~$C$, let $\Theta(C)$ be the collection of~$n$ such that for some finite $E\subset C$ omitting the first~$n$ elements of~$C$, $n\in \Phi(F\cup E)$. Then for all $C\in \+E$ we have $A = \Theta(C)$. The fact that $C\in \+E$ implies that $\Theta(C)\subseteq A$. In the other direction, given $n\in A$, let~$X$ be the set obtained from~$C$ by removing the first~$n$ elements. Then $X\in \+C$, so there is some $Y\in [X]^\w$ such that $(F,Y)$ strongly forces that $n\in \Phi(G)$; in particular, $n\in \Phi(F\cup Y)$, so some finite subset~$E$ of~$Y$ will witness that $n\in \Theta(C)$. 
\end{proof}

As usual, we conclude:

\begin{corollary} \label{thm:a_lot_of_uniformisation_for_Turing}
	If $A\iT B$ then there is a nonempty $\Sigma^1_1(B)$ class of sets $C\in [B]^\w$ satisfying $A \uiT C$. 
\end{corollary}

Below we will use the following:

\begin{corollary} \label{cor:Gandy_basis}
	If $A \iS B$ then there is some $C\in [B]^\w$ such that $A \uiS C$ and $\w_1^C \le \w_1^B$. 
\end{corollary}

The proof is a simple application of the Gandy basis theorem~\cite{Gandy:basis}, which in relativised form says that every nonempty $\Sigma^1_1(B)$ class contains a set~$C$ with $\w_1^{B\oplus C} = \w_1^B$. Note that $\w_1^B\ge \w_1^A$ since $A$ is c.e.\ in~$B$.


\section{Completeness} 
\label{sec:completeness}

In this section we prove \cref{thm:main_final_completeness_theorem}. The argument is elaborate, and so we approach it gradually: we prove two weaker theorems with proofs of increasing complexity, each introducing another element of the final proof. Before we do this, though, we use the basic technique to prove \cref{thm:failure_of_improving_enumerability}\ref{item:failure_left}.

\subsection{Flexibility in constructing introreducible sets} 
\label{sub:flexibility_in_constructing_introreducible_sets}

In the proof of \cref{thm:failure_of_improving_enumerability}\ref{item:failure_left}, we need to construct sets~$A$ and~$B$ with $A \uiS B$ but $C \nleq_T^i B$ for all infinite $C \subseteq A$. By \cref{prop:uie_and_uiS}, this is equivalent to $A \uie B$. We do this by fixing in advance the enumeration functional~$\Phi$ witnessing $A \uie B$. This functional is defined to be ``maximally flexible'': no matter how we have already determined~$A$ and~$B$ so far, we can extend them with freedom to behave with respect to~$\Phi$ as we wish. 

\begin{proof}[Proof of \cref{thm:failure_of_improving_enumerability}\ref{item:failure_left}]
	As discussed, we build sets~$A$ and~$B$ such that $A \uie B$ but that $C\niT B$ for all $C\in [A]^\w$. 

	\medskip
	
	To define the enumeration functional~$\Phi$, we fix a sequence $(X_{i})_{i\in \w}$ of uniformly computable sets satisfying the following: 
	\begin{itemize}
		\item For every finite $E\subset \w$, the set
		\[
			\bigcap_{i\in E} X_i \cap \bigcap_{i\notin E} (\w\setminus X_i)
		\]
		is infinite. 
	\end{itemize}
	For example, we can choose~$X_i$ to be the collection of numbers divisible by the $(i+1)\tth$ prime number. 

	We then define~$\Phi$ as follows: the axioms of~$\Phi$ say that for all~$i$, the singleton subsets of $X_{i}$ suffice to enumerate~$i$. That is, for all~$Z$, $i\in \Phi(Z)$ if and only if $Z\cap X_{i}\ne \emptyset$. 
	
	Now, to get sets~$A$ and~$B$ such that $A\uie B$ via~$\Phi$, it suffices to ensure, for all~$i$, that: 
	\begin{sublemma}
		\item if $i\notin A$ then $B\cap X_{i} = \emptyset$; and 
		\item if $i\in A$ then $B \subseteq^* X_{i}$. 
	\end{sublemma}
	We obtain such~$A$ and~$B$ by forcing with finite conditions: a pair $(\s,\tau)$ determines initial segments of~$A$ and~$B$; it is admissible if it does not violate~(a); and to ensure~(b), for all~$i$ which~$\s$ ensures are in~$A$, all future additions to~$B$ beyond~$\tau$ must be in $X_{i}$. We then also construct a Cohen subset~$D$ of~$B$ and ensure that no infinite subset of~$A$ is computable from both~$B$ and~$D$. 

	Of course, we can construct all of these in one go, so the notion of forcing we describe adds all sets at once. For strings $\s,\tau\in 2^{<\w}$, we use $\s\preceq \tau$ to denote string extension, but we also use the strings to denote the finite sets they determine: we write $i\in \s$ if $i<|\s|$ and $\s(i)=1$; we write $\s\subseteq \tau$ if for all~$i$, $i\in \s$ implies $i\in \tau$, and so on. 

	The forcing notion~$\PP$ consists of triples $(\s,\tau,\rho)\in (2^{<\w})^3$ satisfying:
	\begin{orderedlist}
		\item for all $i\notin \s$, $\tau \cap X_i = \emptyset$;
		\item $\rho\subseteq \tau$. 
	\end{orderedlist}
	If $p = (\s,\tau,\rho)$ is a condition then we write $\s^p = \s$, $\tau^p = \tau$, and $\rho^p = \rho$. A condition~$q$ extends a condition~$p$ if:
	\begin{orderedlist}
		\item $\s^p\preceq \s^q$, $\tau^p\preceq  \tau^q$, and  $\rho^p\preceq  \rho^q$; and
		\item for all $i\in \s^p$, $\tau^q \setminus \tau^p\subset X_i$. 
	\end{orderedlist}
	Observe that this is indeed a partial ordering. A sufficiently generic filter~$\GG$ determines the sets $A[\GG] = \bigcup \left\{ \s^p \,:\,  p\in \GG \right\}$, $B[\GG] = \bigcup \left\{ \tau^p \,:\,  p\in \GG \right\}$ and $D[\GG] = \bigcup \left\{ \rho^p \,:\,  p\in \GG \right\}$. 

	First, observe that the empty condition forces that~$A$ is infinite; for every condition~$p$, $(\s^p\conc 1, \tau^p,\rho^p)$ is an extension of~$p$. The empty condition also forces that~$B$ and~$D$ are infinite: for every condition~$p$, the ``freeness'' property of the sets~$X_i$ allows us to choose some $k>|\tau^p|$ in $X_i$ for all $i\in \s^p$ but outside~$X_i$ for all other~$i$'s, and so $(\s^p,\tau^p\cup \{k\}, \rho^p\cup \{k\})$ is an extension of~$p$. Further, note that for every condition~$p$ and every~$k$, $(\s^p\conc 0^k, \tau^p\conc 0^k, \rho^p)$ is an extension of~$p$. 

	As explained above, the empty condition forces that $A \uie B$ via~$\Phi$. Now, given two Turing functionals~$\Psi$ and~$\Gamma$, we show that the empty condition forces that it is not the case that $\Psi(B)= \Gamma(D)$ are total and equal an infinite subset of~$A$. 

	To see this, let~$p$ be any condition; we may assume that~$p$ forces that:
	\begin{orderedlist}
		\item $\Psi(B)$ is total and infinite; and
		\item $\Gamma(D)$ is total and a subset of~$A$. 
	\end{orderedlist}
	By~(i), there is some $i> |\s^p|$ and a condition $q\le_\PP p$ such that $\Psi(\tau^q,i)\converge = 1$. Let~$k > |\tau^q|,i$; as mentioned, $(\s^p\conc 0^k, \tau^p\conc 0^k,\rho^p)$ is a condition extending~$p$, and so by~(ii), we can find an extension~$r$ of that condition such that $\Gamma(\rho^r,i)\converge = 0$, noting that $i\notin \s^r$.

	Because $\tau^r \cap [|\tau^p|,|\tau^q|) = \emptyset$, we can take a string $\tau^*\succeq \tau^q$ such that $\tau^* = \tau^q \cup \tau^r$. It follows that $\rho^r\subseteq \tau^*$. We also let~$\s^*$ be a string extending~$\s^p$ satisfying $\s^* = \s^q\cup \s^r$. Then $(\s^*,\tau^*,\rho^r)$ is a condition, extends~$p$, and forces that $\Psi(B,i)\ne \Gamma(D,i)$. 
\end{proof}

\subsection{A first completeness result} 
\label{sub:a_first_completeness_result}

We turn to the proof of~\cref{thm:main_final_completeness_theorem}. We show the following:

\begin{proposition} \label{prop:main_completeness_proposition}
	Uniformly, given a linear ordering~$\+L$, we can construct a set~$A$ which is $\Pi^0_1(\+L)$, such that:
	\begin{itemize}
	 	\item If~$\+L$ is well-founded, then~$A$ is uniformly introreducible \emph{relative to~$\+L$}: there is a functional~$\Phi$ such that for every infinite $Z\subseteq A$, $\Phi(Z,\+L) = A$. 
	 	\item If~$\+L$ is ill-founded, then~$A$ is not introreducible relative to~$\+L$.
	 \end{itemize}
\end{proposition}

This suffices:

\begin{proof}[Proof of~\cref{thm:main_final_completeness_theorem}, given~\cref{prop:main_completeness_proposition}]
	Given a linear ordering~$\+L$, let~$A$ be the set given by~\cref{prop:main_completeness_proposition}. Also, uniformly obtain a uniformly introreducible set~$R$ Turing equivalent to~$\+L$ (for example, a Dekker set --- the set of finite initial segments of some real coding~$\+L$). Let $\pi\colon \w\to R$ be an $\+L$-computable bijection; let $B = \pi[A]$. 

	If~$\+L$ is ill-founded, then there is some infinite $S\subseteq A$ such that $A \nle_\Tur S\oplus \+L$; then $\pi[S]$ is an infinite subset of~$B$ which does not compute~$B$ (even relative to~$\+L$). On the other hand, if~$\+L$ is well-founded, then given an infinite $S\subseteq B$, we can first uniformly compute~$R$, and so~$\+L$, and then with~$\+L$ compute~$\pi^{-1}[S]$, then~$A$, and then~$B$. 
\end{proof}

As we build toward the proof of \cref{prop:main_completeness_proposition}, we will first give a proof of a weaker completeness result, introducing important ingredients of the proof. For the rest of the section, to avoid excessive notation, we assume that the linear ordering~$\+L$ is computable; the argument fully rleativises. 

\begin{proposition} \label{prop:completeness1}
	Given a computable, infinite linear ordering~$\+L$ we can effectively obtain a~$\Delta^0_2$ set~$A$ which is uniformly introreducible if and only if~$\+L$ is well-founded. 
\end{proposition}

Thus, the set of  $\Delta^0_2$-indices for introreducible sets is $\Pi^1_1$-complete (with $m$-reductions).

\begin{proof}
	We may assume that the universe of~$\+L$ is~$\w$, but we denote elements of~$\+L$ with lowercase Greek letters. 

	We start with uniformly computable sets $(X_\beta)_{\beta\in \+L}$ and $(Y_F)_{F\in [\w]^{<\w}}$ satisfying:
	\begin{orderedlist}
		\item $(X_\beta)_{\beta\in \+L}$ forms a partition of~$\w$, and so does $(Y_F)_{F\in [\w]^{<\w}}$; 
		\item For every $\beta\in \+L$ and $F\in [\w]^{<\w}$, $X_\beta \cap Y_F$ is infinite. 
	\end{orderedlist}
	For $x\in \w$ we let $r(x)$ be the unique~$\beta$ such that $x\in X_\beta$ (we think of $r(x)$ as a ``rank'' of~$x$). We write $x \le_r y$ if $r(x) \le r(y)$, and $x <_r y$ if $r(x) < r(y)$. We then define the functional~$\Phi$ as follows:
	\begin{itemize}
		\item For a set~$Z$ and $n\in \w$, if there are $x,y\in Z$ such that $n<x<y$ and $x \le_r y$, then we set $\Phi(Z,n) = F(n)$, where~$F$ is the unique finite set such that $y\in Y_F$. 
	\end{itemize}

	We will build a $\Delta^0_2$ set~$A$ such that if~$\+L$ is well-founded then $A \uiT A$ by~$\Phi$. The first thing to notice is that~$\Phi$ only uses positive information from its oracle. This is not surprising, as \cref{prop:uie_and_uiS} implies that $A \uiT B$ if and only if $A\oplus \overline{A} \uie B$. The second thing to notice is that there are many sets~$Z$ on which~$\Phi$ is inconsistent: for some~$n$ we have $\Phi(Z,n)=0$ and $\Phi(Z,n)=1$. In building~$A$, we will need to ensure that $\Phi(A)$ is consistent (and so $\Phi(Z)$ is consistent for all $Z\subseteq A$). Thirdly, we note that given a finite set~$E$ we can compute the partial (possibly multivalued) function~$\Phi(E)$; we have $\dom \Phi(E) \subset \max E$ by the requirement that a convergence of $\Phi(E,n)$ can only be given by elements $x,y\in E$ with $n<x<y$. 

	\medskip
	
	We use a notion of forcing~$\PP= \PP_{\+L}$: the conditions in~$\PP$ are finite sets~$E$ which are \emph{self-consistent} with respect to~$\Phi$, meaning that $\Phi(E)$ is consistent (uni-valued) and agrees with~$E$: for all $n$, if $\Phi(E,n)\converge$ then $\Phi(E,n)= E(n)$. The conditions in~$\PP$ are ordered by end-extension. This notion of forcing is computable. It is not empty: any singleton is an element of~$\PP$, because $\Phi(\{x\},n)\diverge$ for all~$x$ and~$n$.

	The freeness property of the sets~$X_\beta$ and~$Y_F$ implies:

	\begin{claim} \label{clm:extension_in_P}
		If $E\in \PP$ and $y> \max E$ is in $Y_E$ then $E\cup \{y\}\in \PP$. 
	\end{claim}

	\begin{proof}
		Let~$E$ and~$y$ satisfy the assumption. Let $n\in \w$ and suppose that $\Phi(E\cup \{y\},n)\converge=i$; we need to show that $i = (E\cup \{y\})(n)$. Let $w,z\in E\cup \{y\}$ cause the convergence $\Phi(E\cup \{y\},n)\converge=i$: $w,z\in E\cup \{y\}$, $n<w<z$, $w\le_r z$ and $z\in Y_H$ for some finite set~$H$ such that $H(n)=i$. Now $w\le \max E$, so $n< \max E$, so $E(n) = (E\cup \{y\})(n)$. If $z\in E$ then $\Phi(E,n)\converge = i$ and must equal~$E(n)$, as $E\in \PP$. Otherwise $z = y$ and so $H = E$. 
	\end{proof}
	
	\smallskip

	Let~$A= A[\GG]$ be the subset of~$\w$ given by a filter~$\GG$ which is 1-generic for~$\PP$ (meets or avoids each c.e.\ subset of~$\PP$). Such~$A$ can be chosen to be~$\Delta^0_2$, uniformly in~$\+L$. Since every $Y_E$ is infinite, \cref{clm:extension_in_P} implies that~$A$ is infinite. Also, for all $W\subseteq A$ and all~$n$, if $\Phi(W,n)\converge$ then $\Phi(W,n) = A(n)$. 

	\smallskip

	If $\+L$ is well-founded, then $\Phi(Z)$ is total for all~$Z\in [A]^\w$. For let $Z\in [A]^\w$. The increasing enumeration of~$Z$ (or any of its tails) cannot be strictly decreasing for~$<_r$, so there are arbitrarily large $x<y$ in~$Z$ such that $x \le_r y$. Such~$x$ and~$y$ guarantee that $\Phi(Z,n)\converge$ for all $n<x$. 

	\smallskip
	
	Suppose that~$\+L$ is ill-founded. Fix some infinite $\+L$-decreasing sequence $\seq{\beta_i}_{i\in \w}$. The freeness property of the sets~$X_\beta$ and~$Y_E$, together with the 1-genericity of~$\GG$ and \cref{clm:extension_in_P}, imply that for every $\beta\in \+L$, $A\cap X_\beta$ is infinite. Thus, we can choose a sequence $x_0 < x_1 < \cdots$ of elements of~$A$ such that $x_i\in X_{\beta_i}$. Let $Z = \{ x_i\,:\, i\in \w\}$. The point of course is that for $x,y\in Z$ with $x<y$ we have $x>_r y$ and so $\{x,y\}$ does not give a $\Phi$-computation; indeed, $\Phi(Z)$ is defined on no input. 

	To show that $A \nuiT A$ we need to consider functionals other than $\Phi$. Suppose, for a contradiction, that for some functional~$\Psi$ we have $\Psi(Y) = A$ for all $Y\in [A]^\w$. Consider the c.e.\ set 
	\[
		\+D_\Psi = \left\{ F\in \PP \,:\,  (\exists n)\,\, \Psi(F,n)\converge \ne F(n)  \right\}. 
	\]
	If $F\in \+D_\Psi\cap \GG$ then as $F\prec A$ we get $\Psi(A)\ne A$. On the other hand, we show that every condition in~$\GG$ is extended by some condition in~$\+D_\Psi$, which will contradict the 1-genericity of~$\GG$. Let $E\in \GG$; let $y = \min (Z\setminus E)$ (note that $Z\setminus E = Z\cap (\max E,\infty)$ since $Z\subseteq A$ and $E\prec A$).  Let $Y = E\cup Z\setminus \{y\}$. As $Y\subseteq A$, by the assumption on~$\Psi$, $\Psi(Y,y)\converge = A(y) = 1$. There is therefore some finite $F\prec Y$ such that $\Psi(F,y)\converge = 1$. Since $E\prec Y$ we may assume that $E\prec F$. Since $y\notin Y$ we have $\Psi(F,y)\ne F(y)$. So~$F$ is the desired extension of~$E$ in~$\+D_\Psi$, once we show that $F\in \PP$. Let $n\in \w$, and suppose that $\Phi(F,n)\converge$, using a pair $w,z\in F$ with $w< z$ and $w \le_r z$. Then $w\in E$, so $n < \max E$, so as $F\subset A$ we have $\Phi(F,n) = A(n) = E(n) = F(n)$ as required. 
	
	 \smallskip
	 
	 We remark that the proof does \emph{not} show that if $\+L$ is ill-founded then for every functional~$\Psi$, $\GG\cap \+D_\Psi$ is nonempty. After all, that would show that $A \nle_\Tur A$. That~$\+D_\Psi$ is dense along~$\GG$ was proved only under the assumption that~$\Psi$ witnesses that $A \uiT A$. 
\end{proof}

\subsection{Removing uniformity} 
\label{sub:removing_uniformity}

To obtain the desired proposition, we need to improve the proof above in the following ways:
\begin{enumerate}
	\item Make~$A$ co-c.e.\ rather than $\Delta^0_2$; and
	\item Ensure that if $\+L$ is ill-founded, then $A \niT A$ rather than only $A \nuiT A$. 
\end{enumerate}

We now explain how to do the second. 

\begin{proposition} \label{prop:completeness2}
 	Given a computable, infinite linear ordering~$\+L$ we can uniformly obtain~$\Delta^0_2$ set~$A$ such that:
 	\begin{itemize}
 		\item if $\+L$ is well-founded then $A$ is uniformly introreducible; 
 		\item if $\+L$ is ill-founded then~$A$ is not introreducible. 
 	\end{itemize}
 \end{proposition}

\begin{proof}
	We extend the technique of the previous proof. Given~$\+L$, we will use the same sets $X_\beta$	and~$Y_F$ and use them to define the same functional~$\Phi$. We again let~$\PP$ denote the collection of finite sets~$E$ such that $\Phi(E)$ is consistent with~$E$, ordered by end-extension. The set~$A$ will again be 1-generic for~$\PP$.  

	This time, if $\+L$ is ill-founded, we need to find a single set $Z\in [A]^\w$ which does not compute~$A$. Recall that for $x\in \w$, $r(x)$ is the unique~$\beta$ such that $x\in X_\beta$. The set~$Z$ will again be $\{x_0 < x_1 < x_2 < \cdots\}$ where $x_0 >_r x_1 >_r x_2 >_r \cdots$, and so will be defined using some infinite decreasing sequence $(\beta_i)$ from~$\+L$. Because we will need to ensure that $\Psi(Z)\ne A$, this time we will need greater control in the choice of the sequence $(\beta_i)$ and the set~$Z$. Indeed, to actively diagonalise, we will need to work with initial segments of~$Z$ during the construction of~$A$. This is difficult, because any $(\beta_i)$ may be very complicated, certainly out of the grasp of $\emptyset'$. What we do with~$\emptyset'$ is to construct a tree of possible choices for initial segments of~$Z$ for any possible initial segment of $(\beta_i)$. 

	Another new ingredient is a use of an overspill argument to show that certain conditions exist. For that, we will need the well-founded part of~$\+L$ to not be hyperarithmetical. This is standard: given~$\+L$, we can replace it by the linear ordering obtained by the Kleene-Brouwer ordering of the tree of \emph{double descending sequences}, the tree of pairs $(\s,\tau)$ where $|\s|=|\tau|$, $\s$ is a descending sequence in~$\+L$, and $\tau$ is a descending sequence in some fixed Harrison linear ordering (a computable ill-founded linear ordering with no hyperarithmetic descending sequences). If $\+L$ is well-founded then the tree of double descending sequences is well-founded; if $\+L$ is ill-founded then the tree is ill-founded but any path gives an infinite descending sequence in the Harrison linear ordering and so cannot be hyperarithmetic. 

	\medskip
	
	We extend the notion of forcing~$\PP$ to be a 2-step iteration that starts with~$\PP$ and then adds a tree of subsets of the generic for~$\PP$. Combined, we define the notion of forcing~$\QQ$ as follows. The conditions are pairs $(E,\vphi)$, where $E\in \PP$, and $\vphi$ is a finite partial function satisfying:
	\begin{orderedlist}
		\item $\dom \vphi$ is a subtree of $\w^{<\w}$ consisting of descending sequences in~$\+L$. 
		\item For every $\s = \seq{\beta_1,\beta_2,\dots, \beta_k}\in \dom \vphi$, $\vphi(\s)$ is a finite set $F\subseteq E$, $F = \{x_1< x_2 < \cdots < x_k\}$ with $r(x_i)= \beta_i$ for $i=1,2,\dots, k$. 
		\item For $\s\preceq \tau$ in $\dom \vphi$, $\vphi(\tau)$ is an end-extension of $\vphi(\s)$. 
	\end{orderedlist}
	The ordering on~$\QQ$ is given by extension in both coordinates: $(F,\psi)$ extends $(E,\vphi)$ if $E\prec F$, $\dom \vphi\subseteq \dom \psi$ and $\psi\rest{\dom \vphi} = \vphi$. This notion of forcing is computable. 

	\medskip
	
	Let $\GG\subset \QQ$ be a 1-generic filter (computable from~$\emptyset'$). We define
	 $A = A[\GG] = \bigcup \left\{ E \,:\,  (\exists \vphi)\,\,(E,\vphi)\in \GG \right\}$ and
	  $f = f[\GG] = \bigcup \left\{ \vphi \,:\,  (\exists E)\,\,(E,\vphi)\in \GG \right\}$. 
	  Both~$A$ and~$f$ are~$\Delta^0_2$, and $A$ is 1-generic for~$\PP$; it is therefore infinite. Similarly, $\dom f$ is the tree~$T_{\+L}$ of all finite descending sequences in~$\+L$. We give the details. For every $\s\in T_{\+L}$ we show that the collection of conditions $(F,\psi)\in \QQ$ such that $\s\in \dom \psi$ is dense in~$\QQ$, and then appeal to the 1-genericity of~$\GG$. We prove this by induction on $|\s|$. Let $\s\in T_{\+L}$ such that $\s = \s^{-}\conc \seq{\beta}$. By induction, let $(E,\vphi)\in \QQ$ such that $\s^-\in \dom \vphi$. Choose some $y \in X_\beta \cap Y_E$ with $y > \max E$, $y >  \max \range \vphi(\s^-)$; then $E\cup \{y\}\in \PP$ (\cref{clm:extension_in_P}). Define~$\psi$ by extending~$\vphi$ to be defined on~$\s$, letting $\psi(\s) = \vphi(\s^-)\cup \{y\}$. Then the condition $(E\cup \{y\}, \psi)$ is an extension of $(E,\vphi)$ in~$\QQ$.

	  \smallskip
	  
	 If $\+L$ is well-founded, then as in the previous proof, $\Phi$ shows that $A \uiT A$. Suppose then that $\+L$ is ill-founded. Let~$\SSS \subset T_{\+L}$ be the subtree of $\s\in T_{\+L}$ which are entirely contained in the ill-founded part of~$\+L$. We take an infinite descending sequence $(\beta_i)$ in~$\+L$ which is \emph{sufficiently generic} for the partial ordering~$\SSS$. Of course, this partial ordering is complicated (may have degree $\+O$, the complete~$\Pi^1_1$ set) and~$(\beta_i)$ will need a Turing jump or two beyond that. We let $Z = Z [(\beta_i)] = \bigcup \left\{ f(\s) \,:\,  \s\prec (\beta_i) \right\}$. This is an infinite subset of~$A$. We will show that $Z\nge_\Tur A$. 

	 \smallskip

	 Fixing a Turing functional~$\Psi$, we need to show that $\Psi(Z)\ne A$. We show that $\force_{\SSS} \Psi(Z)\ne A$, and appeal to the genericity of $(\beta_i)$. Let
	 \[
	 	\+D = \left\{ \tau\in \SSS \,:\,   \Psi(f(\tau)) \perp A \right\},
	 \]
	 where $\Psi(F)\perp A$ if $\Psi(F,n)\converge \ne A(n)$ for some~$n$. Let $\s\in \SSS$ and suppose that $\s$ forces in~$\SSS$ that $\Psi(Z)$ is total; we find an extension of~$\s$ in~$\+D$. 
	 Let 
	 \[
	 	\+E = \left\{ (F,\theta) \in \QQ \,:\,  (\exists \tau\in \dom \theta)\,\,\tau\in \SSS \andd \tau\succeq \s \andd 
	 		\Psi(\theta(\tau))\perp F	 \right\}.
	 \]
	 We will show that $\GG\cap \+E\ne \emptyset$. This suffices: if $(F,\theta)\in \GG\cap \+E$, witnessed by~$\tau$, then~$\tau$ is an extension of~$\s$ in~$\+D$. 

	 \begin{claim} \label{clm:E_is_dense}
	 	$\+E$ is dense along~$\GG$. 
	 \end{claim}

	 \begin{proof}
		Let $(E,\vphi)\in \GG$; we find an extension of $(E,\vphi)$	in~$\+E$. Let $\beta > \s(0)= \max \range \s$ (we may assume that~$\+L$ has no last element). Since $\GG$ is 1-generic, by extending (and by \cref{clm:extension_in_P}), we may assume that there is some $y> \max E$ in~$X_\beta$ such that $E \cup \{y\}\prec A$. 

		Since $\s\force_{\SSS} ``\Psi(Z)\text{ is total}"$, there is some $\tau\succeq \s$ in~$\SSS$ such that $\Psi(f(\tau),y)\converge$. If $\Psi(f(\tau),y)=0$ then $\GG\cap \+E$ is nonempty. Suppose, then, that $\Psi(f(\tau),y) = A(y) = 1$. 

		Let $F = E\cup f(\tau)$. Note that $\beta\notin \range \tau$ (as $\beta > \s(0) = \tau(0)$ and~$\tau$ is descending in~$\+L$) so $y\notin f(\tau)$; so $y\notin F$. Also, since $E\prec A$ and $f(\tau)\subset A$, $E \prec F$. 

		We first argue that $F\in \PP$. Let~$n$ such that $\Phi(F,n)\converge=i$; let $z,w\in F$ generate this computation, so again $w<z$, $w\le_r z$. Since $F\subset A$, we have $i = A(n)$. For any $s,t\in f(\tau)$ we have $s<t \then s>_r t$, so we cannot have $w,z\in f(\tau)$. Hence $n<w \le \max E$; since $E\prec A$ and $E\prec F$, we have $A(n) = E(n) = F(n)$. 

		Hence $F$ extends~$E$ in~$\PP$. Now define~$\theta$ extending~$\vphi$ by setting $\theta(\rho) = f(\rho)$ for all $\rho\preceq \tau$. Note that~$\theta$ indeed extends~$\vphi$ as~$f$ extends~$\vphi$. To show that $(E,\theta)\in \QQ$ (and so extends $(E,\vphi)$) it suffices to check that for all $\rho\in \dom \theta$, $f(\rho)\subseteq F$. If $\rho\preceq \tau$ then this is by the definition of~$F$. If $\rho\in \dom \vphi$, then as $(E,\vphi)$ is a condition, we have $f(\rho)\subseteq E\subseteq F$. 

		Finally, since $y\notin F$ and $\Psi(\theta(\tau),y)=1$, it is seen that $(F,\theta)\in \+E$, and so is the required extension.
	 \end{proof}
	 
	The set $\+E$ is very complicated (as its definition refers to~$\SSS$, and so to the ill-founded part of~$\+L$), so we cannot directly appeal to the 1-genericity of~$\GG$ and density of~$\+E$ along~$\GG$ to show that $\GG\cap \+E$ is nonempty. We use overspill to overcome this problem. For any $\beta\in \+L$, let 
	\[
		\+E_\beta = \left\{ (F,\theta) \in \QQ \,:\,  (\exists \tau\in \dom \theta)\,\, \tau\succeq \s \andd 
	 		\Psi(\theta(\tau))\perp F \andd \min \range \tau >\beta   \right\}. 
	\]
	If $\beta$ is \emph{well-founded} (is in the well-founded part of~$\+L$), then $\+E\subseteq \+E_\beta$, and so by \cref{clm:E_is_dense}, $\+E_\beta$ is dense along~$\GG$. Each set~$\+E_\beta$ is computable, and so for well-founded~$\beta$ we have $\GG\cap \+E_\beta\ne \emptyset$. The property ``$\+E_\beta \cap \GG \ne \emptyset$'' is an arithmetic property of~$\beta$. Since the well-founded part of~$\+L$ is not even hyperarithmetic, it must be the case that for some ill-founded~$\beta$, $\+E_\beta\cap \GG\ne\emptyset$. But for ill-founded~$\beta$ we have $\+E_\beta\subseteq \+E$, which ends the proof. 
\end{proof}

\subsection{A priority argument} 
\label{sub:a_priority_argument}

As discussed, to prove \cref{prop:main_completeness_proposition}, we now need to show how to make~$A$ co-c.e.\ rather than merely~$\Delta^0_2$. We cannot expect to obtain~$A$ from a 1-generic filter for~$\PP$ since 1-generics do not compute noncomputable c.e.\ sets. However, we don't need the full power of 1-genericity to push through the preceding proof. We carefully list the specific dense sets that we aim to meet, and using a priority argument, we construct a just sufficiently generic set~$A$ which can be made co-c.e.

\begin{proof}[Proof of \cref{thm:main_final_completeness_theorem}]
	Building on the previous proofs, we use all of the ingredients: the sets $X_\beta$ and~$Y_F$, the functional~$\Phi$, and the partial orderings~$\PP$, $\QQ$ and~$\SSS$. We perform a finite-injury priority argument to build~$\GG$ (and hence~$A$ and~$f$). There are two types of requirements:
	\begin{description}

		\item[$P_{\s}$] For every $\s\in T_{\+L}$ (the tree of finite $\+L$-descending sequences), ensure that $\s\in \dom f$. 

		\item[$Q_{\s,\beta,\Psi}$] For every nonempty $\s\in T_{\+L}$, $\beta\in \+L$ and Turing functional~$\Psi$, let \[
		\+E_{\s,\beta,\Psi} = 
			\left\{  (F,\theta) \in \QQ \,:\,  (\exists \tau\in \dom \theta)\,\, \tau\succeq \s \andd 
	 		\Psi(\theta(\tau))\perp F \andd \min \range \tau >\beta     \right\}.
	\]  The requirement aims to make $\GG\cap \+E_{\s,\beta,\Psi}$ nonempty. 
	\end{description}

	We computably order all requirements in order-type~$\w$. We let~$R_e$ denote the $e\tth$ requirement on our list. The stronger requirements are those which appear earlier on the list. We index objects associated with a requirement by its location on the list. For example, if~$R_e$ is the requirement $P_{\s}$ then we write $\s_e = \s$. For brevity, if~$R_e$ is a $Q$-requirement then we write~$\+E_e$ for $\+E_{\s_e,\beta_e,\Psi_e}$. We make sure to order our requirements so that for every~$e$, if~$R_e$ is a $P$-requirement and $|\s_e|\ge 2$ then there is some $k<e$ such that~$R_k$ is a $P$-requirement and $\s_k = (\s_e)^-$ (the sequence obtained from~$\s_e$ by removing the last entry). 

	\medskip
	
	During the construction, at any stage~$s$, a $P$-requirement may be either \emph{waiting} or \emph{satisfied}; whereas a $Q$-requirement may be either \emph{waiting}, \emph{ready} or \emph{satisfied}. 

	At each stage~$s$ we will define a sequence $p_{0,s} \ge_\QQ p_{1,s} \ge_\QQ p_{2,s} \ge_\QQ \cdots \ge_\QQ p_{m(s),s}$ for some $m(s)\in \w$; we will write $p_{e,s} = (E_{e,s}, \vphi_{e,s})$. At stage~$s$ we let 
	\[
		A_s = E_{m(s),s} \cup (\max E_{m(s),s},\infty).
	\]
	
	We ensure that the following holds during the construction at every stage~$s$:
	\begin{sublemma}
		\item If~$R_e$ is not \emph{waiting} at stage~$s$ then $p_{e+1,s}$ is defined (i.e., $e<m(s)$). 
		\item If~$R_e$ is a $P$-requirement and~$p_{e+1,s}$ is defined, then $\s_e\in \dom \vphi_{e+1,s}$.
	\end{sublemma}

	At stage~$s$, a requirement~$R_e$ \emph{requires attention} if it is either \emph{waiting}, or if it is a $Q$-requirement which is \emph{ready}, and there is some condition $q = (F,\theta)\in \+E_{e}$ extending~$p_{e,s}$ which is discovered by stage~$s$ and such that $F\subset A_s$.

	\smallskip
	
	Here is the construction: we start with $n(0)=0$ and $p_{0,0}$ being the empty condition (so $A_0 = \w$). At stage~$s$, let~$R_e$ be the strongest requirement which requires attention at that stage. There is such a requirement since all but finitely many requirements are \emph{waiting} at stage~$s$. Note that $e\le m(s)$, so $p_{e,s}$ is defined. We set $m(s+1)= e+1$. For $k < e$ we let $p_{k,s+1} = p_{k,s}$, and the status of~$R_k$ at stage~$s+1$ is the same as its status at stage~$s$. We define $E_{e+1,s+1}$ and set~$R_e$'s new status as follows.

	\begin{orderedlist}

	 	\item Suppose that $R_e$ is a $P$-requirement. If $\s_e\in \dom \vphi_{e,s}$ then we set $p_{e+1,s+1} = p_{m(s),s}$. Otherwise, write $\s_e = \tau\conc{\gamma}$ with $\gamma \in \mathcal{L}$ (recall that~$\s_e$ is nonempty).  Choose some $y > \max E_{m(s),s}$ in~$X_\gamma\cap Y_{E_{e,s}}$. We set $E_{e+1,s+1} = E_{e,s}\cup \{y\}$. We define $\vphi_{e+1,s+1}$ extending $\vphi_{e,s}$ by setting $\vphi_{e+1,s+1}(\s) = \vphi_{e,s}(\tau)\cup \{y\}$. 
	 	We set~$R_e$ to be \emph{satisfied} at stage~$s+1$. 

	 	\item Suppose that~$R_e$ is a $Q$-requirement, and that~$R_e$ is \emph{waiting} at stage~$s$. Choose some $y> \max E_{m(s),s}$ with $y\in X_\gamma\cap Y_{E_{e,s}}$ for some $\gamma > \max \range \s_e$. We set $p_{e+1,s+1} = (E_{e,s}\cup \{y\}, \vphi_{e,s})$. We declare~$R_e$ to be \emph{ready} at stage~$s+1$. 

	 	\item Suppose that~$R_e$ is a $Q$-requirement, and that~$R_e$ is \emph{ready} at stage~$s$. Let $q = (F,\theta)$ witness that~$R_e$ requires attention at stage~$s$. By adding an element to~$F$, we may assume that $\max F > \max E_{m(s),s}$. Set $p_{e+1,s+1}= q$, and declare~$R_e$ to be \emph{satisfied} at stage~$s+1$. 
	 \end{orderedlist} 

	 Finally, for all $k>e$, we set~$R_k$ to be \emph{waiting}. This completes the description of the construction.   

	 \medskip
	 
	 In beginning the verification, we check that the construction can be performed as described. It is not difficult to verify, by induction on the stages, that~(a) and~(b) above hold at every stage. We also observe that whenever defined, it is indeed the case that $p_{e,s}\in \QQ$. For this, we mostly use \cref{clm:extension_in_P}, to check that $E_{e+1,s+1}\in \PP$ (where~$R_e$ receives attention at stage~$s$). In case that~$R_e$ is a $P$-requirement and $\s_e\notin \dom \vphi_{e,s}$, note that by our ordering of the requirements, and by~(b), $\tau = {\s_{e}}^-$ is indeed in~$\dom \vphi_{e,s}$. Also, $\vphi_{e,s}(\tau)\subseteq E_{e,s}$ and $y > \max E_{m(s),s} \ge \max E_{e,s}$ so $\vphi_{e+1,s+1}(\s_e)$ is of the required form. 

	 \smallskip
	 
	 Next, we observe that for all~$s$, $A_{s+1}\subseteq A_s$. Mainly, this is because $\max E_{m(s+1),s+1} \ge \max E_{m(s),s}$. Also, in~(i) and~(ii), if $E_{m(s+1),s+1}\ne E_{m(s),s}$, then we have $E_{m(s+1),s+1} = E_{e+1,s+1} = E_{e,s}\cup \{y\}$ where $y> \max E_{m(s),s}$ (so $y\in A_s$) and $E_{e,s}\subseteq E_{m(s),s} \subset A_s$. In~(iii) the instructions require $E_{e+1,s+1}\subset A_s$. 

	 \smallskip
	 
	 The construction is finite injury; every requirement is eventually either permanently \emph{satisfied}, or is permanently \emph{ready} but never later \emph{satisfied}. For every~$e$, the sequence $(p_{e,s})_{s\in \w}$ stabilizes, and we denote the stable value by $p_e = (E_e,\vphi_e)$. We have $p_{e+1}\le_\QQ p_e$. We let~$\GG$ be the filter generated by the sequence $(p_e)$, i.e., $\GG = \left\{ q\in \QQ \,:\,  (\exists e)\,\,q\ge_\QQ p_e \right\}$. We let $A = A[\GG]$ and $f = f[\GG]$ defined as in the previous construction. We have $A = \bigcap_s A_s$. 

	 \smallskip
	 
	 Since~$A$ is obtained from a filter of~$\PP$, for every $Z\subseteq A$ and every~$n$ such that $\Phi(Z,n)\converge$ we have $\Phi(Z,n) = A(n)$. 

	 \smallskip
	 
	 By the instruction, and the fact that every requirement is injured only finitely many times, we see that every $P$-requirement is met. Hence $\dom f = T_{\+L}$. This also implies that~$A$ is infinite: For every $\gamma\in \+L$, there are infinitely many sequences $\s\in T_{\+L}$ whose last element is~$\gamma$, and so in fact $A\cap X_\gamma$ is infinite. 

	 \smallskip
	 
	 If~$\+L$ is well-founded, then the argument from the proof of \cref{prop:completeness1} shows that $A\uiT A$ via~$\Phi$. Suppose that~$\+L$ is ill-founded. We define~$\SSS$ as in the previous proof, choose a sufficiently generic $(\beta_i)$ through~$\SSS$, and let $Z = Z[(\beta_i)] = \bigcup \left\{ f(\s) \,:\,  \s\prec (\beta_i) \right\}$. 

	 Let~$\Psi$ be a functional, let~$\s\in \SSS$ which forces that $\Psi(Z)$ is total. As in the previous proof, we find an extension~$\tau$ of~$\s$ in~$\SSS$ such that $\Psi(f(\tau))\perp A$.

	 \begin{claim}
	 	For every well-founded~$\beta\in \+L$, $\GG\cap \+E_{\s,\beta,\Psi}$ is nonempty. 
	 \end{claim}

	 \begin{proof}
	 	Fix such~$\beta$, and let~$e$ be such that $R_e = Q_{\s,\beta,\Psi}$. Suppose, for a contradiction, that $\GG\cap \+E_{e} = \emptyset$. 

	 	Let~$s$ be the last stage at which~$R_e$ is \emph{waiting} and receives attention. If~$R_e$ receives attention after stage~$s$ then it chooses $p_{e+1,s+1}\in \+E_e$; since it is not later initialised (set back to a \emph{waiting} state by some stronger~$R_k$), we have $p_{e+1} = p_{e+1,s+1}$. Hence, our assumption for contradiction implies that~$R_e$ does not receive attention after stage~$s$. We will show that~$R_e$ requires attention after stage~$s$; this will contradict the fact that~$R_e$ is not initialised after stage~$s$. 

	 	For all $t\ge s$ we have $p_{e,t} = p_e$ (so $E_{e,t} = E_e \prec A$). At stage~$s$ we define $E_{e+1,s+1} = E_{e,s} \cup \{y\}$ for some~$y$. Since $\s\force_{\SSS} ``\Psi(Z)\text{ is total}"$, there is some $\tau\in \SSS$ extending~$\s$ such that $\Psi(f(\tau),y)\converge$. As in the previous proof, note that~$y$ was chosen so that $y\notin f(\tau)$ as it is $>_r x$ for all $x\in f(\tau)$. 

	 	Since~$R_e$ does not receive attention after stage~$s$, for all $t>s$ we have $p_{e+1,t} = p_{e+1}$, so $y\in A$. Let $k>e$ be sufficiently large so that $\tau\in \dom \vphi_k$. If $\Psi(f(\tau),y) = 0$ then $p_k\in \+E_e$; hence, by our assumption for contradiction, we have $\Psi(f(\tau),y) = A(y)=1$. 

	 	We again let $F = E_e \cup f(\tau)$. Then $F\subset A$, so $F\subset A_t$ for all~$t$. Again define~$\theta$ by extending $\vphi_{e}$ by setting $\theta(\rho) = f(\rho)$ for all $\rho\preceq \tau$. Then the argument proving \cref{clm:E_is_dense} shows that $(F,\theta)\in \QQ$, extends~$p_e$, and is an element of~$\+E_e$. At a large enough stage $t>s$ we discover this condition, and as $F\subset A_t$, we see that~$R_e$ requires attention at stage~$t$, which leads to the desired contradiction.
	 \end{proof}
	 
	 We end the proof as above: since the well-founded part of~$\+L$ is not hyperarithmetic, $\GG$ intersects $\+E_\beta$ for some ill-founded~$\beta$, which gives the desired extension of~$\s$ in~$\SSS$; the genericity of $(\beta_i)$ yields $\Psi(Z)\ne A$. This completes the proof of \cref{prop:main_completeness_proposition}, and so of~\cref{thm:main_final_completeness_theorem}. 
\end{proof}


\section{Obtaining introreducible subsets} 
\label{sec:improving_enumerability}

In this section we find introreducible subsets of given introenumerable sets. We will start with a proof of \cref{thm:improving_enumerability_one_sided}\ref{item:noncollapsing}, and then elaborate on its technique.

\subsection{When all ordinals are computable} 
\label{sub:when_all_ordinals_are_computable}

Suppose that $A \uie B$, via some enumeration functional~$\Phi$. How can we tell if a given~$n$ is in~$A$? Consider the tree of finite subsets~$E$ of~$B$ which do not enumerate~$n$ using~$\Phi$, ordered by end-extension. If $n\in A$ then this subtree must be well-founded, and so every subset has an ordinal rank. Otherwise, it is the full infinitely-splitting tree of all finite subsets of~$B$. In the next proposition, we show that if the ordinal ranks are all computable, then with the aid of a sufficiently fast-growing function and a subset of~$B$, we can use this to compute whether~$n$ is in~$A$ or not.

\begin{proposition} \label{prop:noncollapsing}
	Suppose that $A \uie B$ and $\w_1^B = \wock$. Then there is some $C\in [B]^\w$ which is $\Delta^1_1(B)$ and such that $A \uiT C$. 
\end{proposition}

Using this proposition we can prove Theorem \ref{thm:improving_enumerability_one_sided}, which weakens the hypothesis $A \uie B$ to the non-uniform $A \ie B$: if $A \ie B$ and $\omega_1^B = \wock$, then there is an infinite $C \subseteq B$ such that $A \uiT C$.

\begin{proof}[Proof of \cref{thm:improving_enumerability_one_sided}\ref{item:noncollapsing}, given \cref{prop:noncollapsing}]
	Suppose that $A \iS B$ and $\w_1^B = \wock$. By \cref{cor:Gandy_basis}, find some $D\in [B]^\w$ with $\w_1^D \le \w_1^B$ (so $\w_1^D = \wock$) and $A \uiS D$. By \cref{prop:uie_and_uiS}, $A \uie D$. Now apply \cref{prop:noncollapsing} to~$A$ and~$D$. 
\end{proof}

This proposition also implies that if~$A$ is uniformly introenumerable and $\w_1^A = \wock$, then $A$ has a uniformly introcomputable subset. For if~$A$ is uniformly introenumerable, then by~\cref{prop:uie_and_uiS}, $A\uie A$; apply \cref{prop:noncollapsing} to~$A$ and~$A$ to get some $C\in [A]^\w$ such that $A \uiT C$ and $C\in \Delta^1_1(A)$; now apply~\cref{lem:3.9}. To get \cref{thm:improving_enumerability}, we will follow this argument, but will need to replace the assumption $\w_1^B = \wock$ in \cref{prop:noncollapsing} with the assumption that~$A$ is introenumerable, which will take some work.

\begin{proof}[Proof of \cref{prop:noncollapsing}]
	Suppose that $A \uie B$ via an enumeration functional~$\Phi$. By a standard time-trick, we assume that the map $E\mapsto \Phi(E)$ (for finite sets~$E$) is computable (if at a late stage~$s$ we see that $n\in \Phi(E)$, we instead enumerate~$n$ with use all the extensions of~$E$ which have elements greater than~$s$). 

	\smallskip

	We introduce notation for the rank of a finite subset of~$E$ that will be modified later in the paper. For $n\in \w$, $E\in [B]^{<\w}$ and ordinal~$\alpha$ we define the relation $\erank{n}{E} \ge \alpha$ by recursion on~$\alpha$. For all $n$ and~$E$ we have $\erank{n}{E} \ge 0$. Suppose that $\alpha\ge 1$. 
	\begin{itemize}
		\item $\erank{n}{E} \ge 1$ if $n\notin \Phi(E)$. 
		\item For $\alpha>1$, $\erank{n}{E}\ge \alpha$ if for all~$\beta<\alpha$ there is some $y\in B$, $y>\max E$ such that $\erank{n}{E\cup \{y\}}\ge \beta$. 
	\end{itemize}
	We then let $\erank{n}{E} = \alpha$ if $\erank{n}{E} \ge \alpha$ but $\erank{n}{E}\nge \alpha+1$; if $\erank{n}{E} \ge \alpha$ for every ordinal $\alpha$, then we write $\erank{n}{E} = \infty$. 

	To unpack: $\erank{n}{E} = 0$ if and only if $n\in \Phi(E)$; $\erank{n}{E} = \beta+1$ if and only if for every $y\in B$ greater than~$\max{E}$ we have $\erank{n}{E\cup\{y\}} \le \beta$ but for some such~$y$ we have $\erank{n}{E\cup\{y\}} = \beta$; for a limit ordinal~$\gamma$, $\erank{n}{E} = \gamma$ if and only if for all $y\in B$, $y>\max E$ we have $\erank{n}{E\cup\{y\}}<\gamma$ but $\{\erank{n}{E\cup\{y\}}\,:\, y\in B, y>\max E\}$ is unbounded in~$\alpha$.

	\smallskip

	We presented this inductive definition, because it is the one that will be modified later; however we can also formalise this as a tree rank. For each $n$ (thought of as either an element of $A$ or its compliment), let~$T_n$ be a subtree of the tree of finite subsets $E$ of~$B$, ordered by end-extension, with the leaves of~$T_n$ being the minimal sets~$E$ (again under end-extension) satisfying $n\in \Phi(E)$. If $E\in T_n$ (in particular, if $n\notin \Phi(E)$ or $E$ is a leaf of $T_n$) then~$\erank{n}{E}$ is the rank of~$E$ on the tree~$T_n$. Using the fact that $\Phi$ witnesses that $A \uie B$, it follows that~$T_n$ is well-founded if and only if $n\in A$. For completeness, we give an argument using the inductive definition.

	In fact, we observe that for all~$n$, $n\in A$ if and only if $\erank{n}{\emptyset}<\infty$ if and only if for all $E\in [B]^{<\w}$, $\erank{n}{E}< \infty$. If~$n\notin A$ then by induction on~$\alpha$ we show that for all $E\in [B]^{<\w}$, $\erank{n}{E}\ge \alpha$; the main step is showing this for $\alpha=1$, which follows from $n\notin \Phi(B)$. In the other direction, suppose that $\erank{n}{E}=\infty$ for some $E\in [B]^{<\w}$; then there is some $y>\max E$ in~$B$ such that $\erank{n}{E\cup \{y\}} = \infty$, as $B\setminus E$ is a set and not a proper class. Repeating, we build an infinite set~$Z$ such that $n\notin \Phi(Z)$ (as $\erank{n}{F} = \infty >0$ for all $F\prec Z$) which shows that $n\notin A$. 

	\medskip
	
	Since $A$ is c.e.\ in~$B$, we have $A\in \Delta^1_1(B)$, in particular $\w_1^{A\oplus B} = \w_1^B = \wock$. For every $n\in A$, $\erank{n}{\emptyset} < \wock$ as it is the tree rank of the tree~$T_n$, which is $B$-computable. Moreover, the set $\left\{ \erank{n}{\emptyset} \,:\, n\in A  \right\}$ is bounded below~$\wock$. To see this, let $T_A$ be the tree which is the disjoint sum of the trees~$T_n$ for $n\in A$ (add a root below the roots of all of these trees); the tree~$T_A$ is $B'$-computable, and so its rank is $B$-computable. We fix some computable ordinal~$\delta$ such that for all $n\in A$, $\erank{n}{\emptyset}\le  \delta$. 

	Further, we fix some computable presentation of~$\delta+1$ which is \emph{notation-like}, and identify every ordinal $\beta \le \delta$ with the natural number coding this ordinal in our computable presentation. Notation-like means: 
	\begin{orderedlist*}
		\item The set of limit $\gamma \le \delta$ is computable; 
		\item The successor function on~$\delta$ is computable. 
	\end{orderedlist*}
	This implies that the function taking a successor ordinal $\beta\le \delta$ to its predecessor is also computable, and that for every limit $\gamma\le \delta$ we can (uniformly) compute a cofinal $\w$-sequence $\gamma[0]< \gamma[1] < \cdots$ in~$\gamma$. For successor ordinals $\gamma \le \delta$ let $\gamma[k] = \gamma-1$ for all~$k$. 

	\smallskip
	
	We observe that after fixing the presentation of~$\delta$, the function~$\lambda$ is $\Delta^1_1(B)$; this is because this is the ranking function for the tree~$T_A$. 
	
	\medskip
	
	We say that a function $f\colon \w\to \w$ is a \emph{deficiency} function if:  
	\begin{itemize}
		\item For every nonzero $\gamma\le \delta$, $n\in A$ and $E\in [B]^{<\w}$,
		 \[
		 \erank{n}{E}<\gamma \Then \erank{n}{E} \le \gamma[f(\gamma,n,E)]. 
		 \]
	\end{itemize}
	Since~$\lambda$ is $\Delta^1_1(B)$, there is a $\Delta^1_1(B)$ deficiency function. Further (and this is the main point), any function majorising a deficiency function is itself a deficiency function. Now we choose some $C\in [B]^\w$ such that $p_C$ (the principal function of~$C$) is a deficiency function; since there is a $\Delta^1_1(B)$ deficiency function, we can choose~$C$ to be $\Delta^1_1(B)$. We show that $A \uiT C$. Indeed, we show that if $D\in [B]^\w$ and~$p_D$ is a deficiency function, then $A\le_\Tur D$ uniformly; this holds for every $D\in [C]^\w$. 

	\smallskip
	
	Fix such~$D$. Given $n\in \w$, by recursion on $i=0,1,\dots$ we compute a decreasing sequence of ordinals $\beta_0 > \beta_1 > \cdots$ as follows:
	\begin{itemize}
		\item $\beta_0 = \delta$; 
		\item if $\beta_i>0$ then $\beta_{i+1} = \beta_i[p_D(\beta_i,n,D_{i+1})]$, where~$D_i$ consists of the first~$i$ elements of~$D$. 
	\end{itemize}
	This sequence must of course halt at some~$j$ such that $\beta_j = 0$. We claim that $n\in A \Iff n\in \Phi(D_j)$, which shows how to compute~$A$ from~$D$ (recall that we assumed that the relation $n\in \Phi(E)$ is computable rather than merely c.e.). 

	Certainly, if $n\notin A$ then $n\notin \Phi(D)$ so $n\notin \Phi(D_j)$. Suppose, then, that $n\in A$. We want to show that $n\in \Phi(D_j)$, equivalently, that $\erank{n}{D_j} = 0$. By induction on $i\le j$ we show that $\erank{n}{D_i}\le \beta_i$. 

	For $i=0$, by choice of~$\delta$, we have $\erank{n}{D_0} = \erank{n}{\emptyset} \le \delta = \beta_0$. 

	For the induction step, suppose that $i<j$ and that $\erank{n}{D_i}\le \beta_i$. We note that $D_{i+1}$ is a proper extension of~$D_i$, and so $\erank{n}{D_{i+1}} < \erank{n}{D_i}$, so $\erank{n}{D_{i+1}} < \beta_i$. Now since~$p_D$ is a deficiency function, $\erank{n}{D_{i+1}} < \beta_i$ implies $\erank{n}{D_{i+1}} \le \beta_i[p_D(\beta_i,n,D_{i+1})] = \beta_{i+1}$. 
	This ends the proof. 
\end{proof}

\subsection{The general case} 
\label{sub:the_general_case}

As we discussed, we want to prove \cref{prop:noncollapsing}, but remove the assumption that $\w_1^B = \wock$. By \cref{thm:failure_of_improving_enumerability}\ref{item:failure_right}, this cannot always be done. It is instructive to think about what part of the proof of \cref{prop:noncollapsing} fails without the assumption. What was special about the ordinal~$\delta$ was that it is \emph{computable}, rather than merely $B$-computable. This was implicitly used in the computation process we described using a sufficiently sparse set $D\in [B]^\w$: it has access to our computable copy of~$\delta$. Since there is no reason to assume that~$D$ computes~$B$, we may have $\w_1^D < \w_1^B$; but even if $\w_1^D \ge \w_1^B$ and $\delta$ is $D$-computable, different subsets~$D$ of~$B$ may not be able to agree on a single copy of~$\delta$. The notion of deficiency function was highly dependent on the choice of cofinal sequences $(\beta[k])$, and so of our copy of~$\delta$. 	

Instead of the impossible, we will prove:

\begin{proposition} \label{prop:improving:main}
	Suppose that $A \uie B$. Then there are $X\in [A]^\w$ and $Y\in [B]^\w$ such that $X \uiT Y$. Further, we can choose $X,Y\in \Delta^1_1(B)$. 
\end{proposition}

Let us explain why this suffices. The proposition implies:

\begin{proposition} \label{prop:1.6_strengthened}
	Suppose that~$A$ is introenumerable and that $A \uie B$ . Then there is some $C\in [B]^\w$ such that $C\in \Delta^1_1(B)$ and $A\uiT C$. 
\end{proposition}

Note that this implies \cref{thm:improving_enumerability_one_sided}\ref{item:introenumerable}; we use the same argument proving \cref{thm:improving_enumerability_one_sided}\ref{item:noncollapsing} from \cref{prop:noncollapsing}, except that we do not need the full power of \cref{cor:Gandy_basis}; \cref{prop:uniformization_theorem_for_c.e.} suffices.

\begin{proof}[Proof of \cref{prop:1.6_strengthened}, assuming \cref{prop:improving:main}]
	By \cref{prop:improving:main}, let $X\in [A]^\w$ and $Y\in [B]^\w$ such that $X \uiT Y$ and such that $X,Y\in \Delta^1_1(B)$. Since~$A$ is introenumerable, it is c.e.\ in~$X$. Let~$g$ be the modulus function for some $X$-computable enumeration of~$A$; so~$A$ is computable uniformly given~$X$, and any function majorising~$g$. Note that $g\le_\Tur X'$ and so it is $\Delta^1_1(B)$. Now let $C\in [Y]^\w$ sufficiently sparse so that $p_C \ge g$; since $g,Y\in \Delta^1_1(B)$, we can choose such $C\in \Delta^1_1(B)$. Since $X\uiT C$, given a subset of~$C$ we can compute both~$X$ and a function majorising~$g$, and so compute~$A$. 
\end{proof}

We indicated above how this implies our main theorem, but for neatness, let us state the proof. Recall that Theorem \ref{thm:improving_enumerability} says that every uniformly introenumerable set has an infinite uniformly introreducible subset.

\begin{proof}[Proof of \cref{thm:improving_enumerability}, assuming \cref{prop:improving:main}]
	Let~$A$ be uniformly introenumerable. By~\cref{prop:uie_and_uiS}, $A\uie A$. Apply \cref{prop:1.6_strengthened} to~$A$ and~$A$ to get some $B\in [A]^\w$ such that $A \uiT B$ and $B\in \Delta^1_1(A)$; now apply~\cref{lem:3.9} to get a uniformly introcomputable subset of~$A$.
\end{proof}

\medskip

Let us now briefly explain how we can modify the proof of \cref{prop:noncollapsing}  into a proof of \cref{prop:improving:main}. We give up on trying to obtain a common copy of~$\delta$, but we still want to devise a decision procedure which resembles the one we gave in the proof of \cref{prop:noncollapsing}. Given a very sparse set $D\in [B]^\w$ and $n$, suppose that we have some other~$m$ which we know is in~$A$, and suppose that further, we know that $\erank{n}{\emptyset}$ is either~$\infty$ or less than $\erank{m}{\emptyset}$. Again let $D_i$ be the first $i$ elements of~$D$. We could then try to perform a comparison. The desirable situation is that if~$p_D$ grows sufficiently fast, then it will help us compute a sequence $F_0 \prec F_1 \prec F_2 \prec \cdots$ of finite subsets of~$D$ such that for all~$i$, if $n\in A$ then $\erank{n}{D_i} \le \erank{m}{F_i}$. Then we could keep going until we find some~$j$ such that $m\in \Phi(F_j)$, which must eventually happen as $m \in A$; and then $n\in A$ if and only if $n\in \Phi(D_j)$. 

What this seems to require is a monotonicity of rank with 1-step extensions: if $F\subset D$, $y,z\in D$ and $\max F < y < z$, then $\erank{m}{F\cup\{y\}} \le \erank{m}{F\cup \{z\}}$. In general, there is no reason to believe that this is the case, and so we will need to first thin~$B$ out sufficiently to get a subset for which this is the case. The process would appear simple: if this fails, throw~$z$ out, assuming we chose~$y$ with $\erank{m}{F\cup \{y\}}$ sufficiently large so that we're not forced to throw out too many elements like~$z$. However, this single step now affects $\erank{m}{E}$ even for some $E\prec F$, so we need to be careful about this winnowing process. The way to do it is rank-by-rank, rather than say from the root of the tree upwards. 

In turn, what this implies is that this process of winnowing must be transfinite. Which introduces a whole new complication: why would we be left with an infinite set at a limit stage of the process? After all, it is very easy to devise a decreasing sequence $B_0 \supset B_1 \supset B_2 \supset \cdots$ with $B_\w = \bigcap_k B_k = \emptyset$. However we can always find an infinite set $B_\w\subseteq^* B_k$ for all~$k$. The fact that our sequence of subsets ignores finite differences, implies that we need to modify our ranking function as well, so that it too ignores finite differences. 

Those are some of the main ideas; we now give the proof. 

\begin{proof}[Proof of \cref{prop:improving:main}]
	We start with sets~$A$ and~$B$ satisfying $A \uie B$, via some enumeration functional~$\Phi$. As above, we assume that the relation $n\in \Phi(E)$ is computable. 

	\medskip
	
	For sets $Z\in [B]^\w$, $n\in \w$ and $E\in [B]^{<\w}$, we define a rank $\frk{Z}{n}{E}$ by defining by recursion on ordinals~$\alpha$ the relation $\frk{Z}{n}{E}\ge \alpha$:
	\begin{itemize}
		\item $\frk{Z}{n}{E}\ge 0$ for all~$n$ and~$E$.
		\item $\frk{Z}{n}{E} \ge 1$ if $n\notin \Phi(E)$.
		\item For $\alpha>1$, $\frk{Z}{n}{E}\ge \alpha$ if for all $\beta<\alpha$ there are infinitely many $y\in Z$ such that $\frk{Z}{n}{E\cup \{y\}} \ge \beta$. 
	\end{itemize}
	As above, $\frk{Z}{n}{E} = \alpha$ if $\frk{Z}{n}{E}\ge \alpha$ but $\frk{Z}{n}{E}\nge \alpha+1$; if $\frk{Z}{n}{E}\ge \alpha$ for all~$\alpha$ then we write $\frk{Z}{n}{E} = \infty$. Note that we do not require $E\subset Z$ to define $\frk{Z}{n}{E}$, only $E\subset B$.

	\begin{claim} \label{clm:modified_rank_and_A} 
		For all~$n$ and all $Z\in [B]^\w$, $n\in A$ if and only if $\frk{Z}{n}{\emptyset} <\infty$. 
	\end{claim}
	
	\begin{proof}
		Similar to the argument above. If $n\notin A$ then by induction on~$\alpha$ we see that for all~$Z$,~$n$ and~$E$ we have $\frk{Z}{n}{E}\ge \alpha$. We use the fact that~$Z$ is infinite. If $\frk{Z}{n}{E} = \infty$ then we inductively build an infinite subset of~$B$ which does not enumerate~$n$, so $n\notin A$. 

		For the second direction, we can alternatively show that $\frk{Z}{n}{E} \le \erank{n}{E}$, where the latter was defined in the previous proof: by induction on~$\alpha$, we show that if $\frk{Z}{n}{E}\ge \alpha$ then  $\erank{n}{E}\ge \alpha$. This clearly holds for $\alpha=0,1$, as the conditions for these do not involve~$Z$; For $\alpha>1$, we use the fact that an infinite set contains an element $y> \max E$. It follows that if $\frk{Z}{n}{E} = \infty$ then $\erank{n}{E} = \infty$ and so $n\notin A$. 
	\end{proof}

	The fact that $\frk{Z}{n}{E}\le \erank{n}{E}$ shows that 
	 \[
	 \left\{  \frk{Z}{n}{E} \,:\, Z\in [B]^\w, n\in A \andd E\in [B]^{<\w}  \right\}
	 \]
	 is bounded below $\w_1^B$. We fix some $\delta^*<\w_1^B$ which bounds all of these ordinals, and for computation purposes, we fix some $B$-computable copy of~$\delta^*$. The modified rank can be computed hyperarithmetically in~$B$:

	\begin{claim} \label{clm:computing_rank}
		For all $Z\in [B]^\w$, the relations $\frk{Z}{n}{E}\ge \alpha$ are uniformly computable from $(Z\oplus B)^{(2\alpha+2)}$, and so the function $(n,E)\mapsto \frk{Z}{n}{E}$ is $\Delta^1_1(Z\oplus B)$. 
	\end{claim}
	
	We also need:

	\begin{claim}	\label{clm:rank_and_subset}
		If $Z,W\in [B]^\w$ and $W\subseteq^* Z$ then for all~$n$ and $E\in [B]^{<\w}$, $\frk{W}{n}{E}\le \frk{Z}{n}{E}$. 
	\end{claim}

	In particular, if $Z =^* W$ then $\frk{Z}{n}{E} = \frk{W}{n}{E}$ for all~$n$ and~$E\in [B]^{<\w}$.

	\begin{proof}
		By induction on~$\alpha$ we show that $\frk{W}{n}{E} \ge \alpha$ implies $\frk{Z}{n}{E}\ge \alpha$. The case $\alpha=1$ does not depend on~$W$ or~$Z$. For $\alpha>1$, assuming the claim holds for all $\beta<\alpha$, if $\frk{W}{n}{E}\ge \alpha$ then for all $\beta<\alpha$ there are infinitely many $y\in W$ such that $\frk{W}{n}{E\cup \{y\}}\ge \beta$; all but finitely many of these~$y$'s are in~$Z$ as well, and by induction, $\frk{Z}{n}{E\cup \{y\}} \ge \beta$ for these $y$'s; so $\frk{Z}{n}{E}\ge \alpha$. 
	\end{proof}
	
	Given $Z\in [B]^\w$, when passing to a subset $\hat{Z}$ of $Z$, the ranks $\frk{\hat{Z}}{n}{E}$ might go down. This would cause an issue where we shrink $Z$ to dominate a given function, but in shrinking $Z$ we change the ranking and so we also change the function we want to dominate; we might thus be perpetually chasing our tail. We now introduce the main tool we use, which is a property of such as set $Z$, called being \textit{rank-minimal}, which implies that the ranks do not decrease when passing to subsets.
	\begin{itemize}
		\item We say that $Z\in [B]^\w$ is \emph{rank-minimal} if for every~$n$, $E\in [B]^{<\w}$ and~$\alpha$, if $\frk{Z}{n}{E}\ge \alpha$ then for all $\beta<\alpha$, for all but finitely many $y\in Z$ we have $\frk{Z}{n}{E\cup \{y\}} \ge \beta$. 
	\end{itemize}

	An equivalent condition is: for all $n\in A$ and $E\in [B]^{<\w}$,
	 \[
	 \frk{Z}{n}{E} = \lim_{y\in Z} \left(\frk{Z}{n}{E\cup \{y\}} +1\right). 
	 \]
	
	 \medskip

	 Our first task is the construction of a rank-minimal subset of~$B$. As indicated above, such a set will be approximated from above by a transfinite process. We therefore need the following finer concept: For any ordinal~$\gamma$, we say that $Z\in [B]^\w$ is \emph{rank-minimal up to~$\gamma$} if for all $\alpha\le \gamma$, for every~$n$ and $E\in [B]^{<\w}$, if $\frk{Z}{n}{E}\ge \alpha$ then for all~$\beta<\alpha$, for all but finitely many $y\in Z$ we have $\frk{Z}{n}{E\cup\{y\}}\ge \beta$. 

	\begin{claim} \label{clm:star_subsets_of_restricted_rank_minimal}
		Let~$\gamma$ be an ordinal. If $Z\in [B]^\w$ is rank-minimal up to~$\gamma$ then for all $W\in [B]^\w$ such that $W\subseteq^* Z$, for all $\alpha\le \gamma$, for all~$n$ and all $E\in [B]^{<\w}$, $\frk{W}{n}{E}\ge \alpha$ if and only if $\frk{Z}{n}{E}\ge \alpha$. 
	\end{claim}

	\begin{proof}
		 By \cref{clm:rank_and_subset}, for all~$n$ and $E\in [B]^{<\w}$ we have $\frk{W}{n}{E}\le \frk{Z}{n}{E}$. So by induction on $\alpha\le \gamma$ we show that for all~$n$ and~$E$, if $\frk{Z}{n}{E}\ge \alpha$ then $\frk{W}{n}{E}\ge \alpha$. As above, for $\alpha=0$ and~$\alpha=1$ the definition of $\frk{X}{n}{E}\ge \alpha$ is independent of~$X$. For $\alpha>1$, take any~$n$ and~$E$ and suppose that $\frk{Z}{n}{E}\ge \alpha$. Let $\beta<\alpha$. Since $Z$ is rank-minimal up to~$\alpha$, for all but finitely many $y\in Z$ we have $\frk{Z}{n}{E\cup \{y\}} \ge \beta$; by induction, for all such~$y$ in~$W$, $\frk{W}{n}{E\cup\{y\}}\ge \beta$. Thus $\frk{W}{n}{E}\ge \alpha$, as required. 
	\end{proof}
	
	It follows that if $Z$ is rank-minimal up to~$\gamma$, then every $W\subseteq^* Z$ in $[B]^\w$ is also rank-minimal up to~$\gamma$. Also note that the choice of $\delta^*$ bounding $\frk{B}{n}{\emptyset}$ for all $n\in A$ implies that $Z\in [B]^\w$ is rank-minimal if and only if it is rank-minimal up to~$\delta^*$.

	\smallskip

	We can now dispense with the existence proof.

	\begin{claim} \label{clm:minimal:existence}
		There is a rank-minimal set $R\in [B]^\w$ which is~$\Delta^1_1(B)$. 
	\end{claim}

	\begin{proof}
		We define a sequence $(R_\gamma)_{\gamma\le \delta}$ which is decreasing by $\subseteq^*$ ($\alpha<\gamma$ implies $R_\alpha\supseteq^* R_\gamma$) such that each~$R_\gamma$ is rank-minimal up to~$\gamma$. We can start with $R_1 = B$, since every infinite subset of~$B$ is rank-minimal up to~1. At the end we let $R = R_{\delta^*}$; as we just discussed, it being rank-minimal up to $\delta^*$ implies that it is rank-minimal. 

		\medskip
		
		We first consider the limit case. Suppose that~$\gamma$ is a limit ordinal and that~$R_\alpha$ has been defined for all~$\alpha<\gamma$. Since~$\gamma$ is countable, we can choose some infinite $R_\gamma$ such that $R_\gamma\subseteq^* R_\alpha$ for all $\alpha<\gamma$: fix a cofinal $\w$-sequence $(\gamma[i])$ in~$\gamma$ and let $A = \{a_0 < a_1 < \dots\}$ where $a_k\in \bigcap_{i\le k} A_{\gamma[i]}$. The condition $R_\gamma\subseteq^* R_\alpha$ for all $\alpha<\gamma$ ensures that~$R_\gamma$ is rank-minimal up to~$\gamma$: for all~$\alpha<\gamma$, the fact that $R_\gamma\subseteq^* R_\alpha$ implies that $R_\gamma$ is rank-minimal up to~$\alpha$. But rank-minimality up to a level is a continuous notion. Let $n\in \w$ and $E\in [B]^{<\w}$ and suppose that $\frk{R_\gamma}{n}{E}\ge \gamma$. Let $\beta<\gamma$. Then $\frk{R_{\gamma}}{n}{E} \ge \beta+1$, and so the condition $\frk{R_\gamma}{n}{E\cup \{y\}}\ge \beta$ for almost all $y\in R_\gamma$ holds because $R_\gamma$ is rank-minimal up to $\beta+1$. 

		\medskip
		
		We next consider the successor case. Suppose that $R_{\gamma}$ has been constructed; we find $R_{\gamma+1}\subseteq^* R_\gamma$, which is rank-minimal up to~$\gamma+1$. To do this, list all the pairs $(n,E)\in \w\times [B]^{<\w}$ satisfying $\frk{R_\gamma}{n}{E} \ge \gamma+1$. We define a decreasing sequence of sets $R_\gamma = C_0 \supseteq C_1 \supseteq C_2 \supseteq \cdots$ as follows: given $C_k$, let $(n,E)$ be the $k\tth$ pair on our list. 
		\begin{itemize}
			\item If $\frk{C_k}{n}{E}\ge \gamma+1$, then we let $C_{k+1}$ be the (infinite) set of $y\in C_k$ such that $\frk{C_k}{n}{E\cup \{y\}}\ge \gamma$. 
			\item Otherwise, let $C_{k+1} = C_k$. 
		\end{itemize}
		We then choose $R_{\gamma+1}\subseteq^* C_k$ for all~$k$ (we can also get $R_{\gamma+1}\subseteq R_\gamma$ rather than just $R_{\gamma+1}\subseteq^* R_\gamma$ if we want, but this is unimportant). Let us check that~$R_{\gamma+1}$ is indeed rank-minimal up to~$\gamma+1$. Since $R_{\gamma+1}\subseteq^* R_\gamma$, we know that it is rank-minimal up to~$\gamma$. Let $n\in \w$ and $E\in [B]^\w$ such that $\frk{R_{\gamma+1}}{n}{E} \ge \gamma+1$. By \cref{clm:rank_and_subset}, $\frk{R_\gamma}{n}{E}\ge \gamma+1$ as well, and so the pair $(n,E)$ was considered at some step~$k$. Since $R_{\gamma+1}\subseteq^* C_k$, we have $\frk{C_k}{n}{E}\ge \gamma+1$; so for all $y\in C_{k+1}$ (and so for almost all $y\in R_{\gamma+1}$), $\frk{C_k}{n}{E\cup\{y\}}\ge \gamma$. Since~$R_\gamma$ is rank-minimal up to~$\gamma$ and $C_{k}\subseteq R_\gamma$, $C_k$ is also rank-minimal up to~$\gamma$, and so by \cref{clm:star_subsets_of_restricted_rank_minimal}, for almost all $y\in R_{\gamma+1}$, $\frk{R_{\gamma+1}}{n}{E\cup \{y\}}\ge \gamma$, as required. 

		\medskip
		 
		It remains to show that the sequence $(R_\gamma)_{\gamma \le \delta^*}$ can be chosen so that the entire sequence is $\Delta^1_1(B)$ (and so certainly $R = R_{\delta^*}$ is $\Delta^1_1(B)$). We can calculate a $B$-computable sequence of ordinals $(\epsilon_\gamma)_{\gamma\le \delta^*}$ such that for all $\gamma \le \delta^*$, $R_\gamma \le_\Tur B^{(\epsilon_\gamma)}$, uniformly in~$\gamma$. For limit~$\gamma$, we can let $\epsilon_\gamma = \sup_{\alpha<\gamma} \epsilon_\alpha$; this is because uniformly in~$\gamma$, we can $B$-computably choose a cofinal $\w$-sequence $(\gamma[k])$ in~$\gamma$; uniformly in~$k$, $B^{(\epsilon_\gamma)}$ computes $B^{(\epsilon_{\gamma[k]})}$ and so $R_{\gamma[k]}$ and the construction of $R_\gamma\subseteq^* R_{\gamma[k]}$ for all~$k$ can proceed computably from the sequence $(R_{\gamma[k]})_{k\in \w}$. For the successor case, if $\epsilon_\gamma$ has been determined, then we can find $\epsilon_{\gamma+1}$ as the limit of an increasing sequence $(\xi_k)_{k<\w}$ with~$B^{(\xi_k)}$ computing~$C_k$; we need an extra $2\xi_k + 4$ jumps to compute the relations $\frk{C_k}{n}{E}\ge \gamma+1$ and $\frk{C_k}{n}{E\cup \{y\}}\ge \gamma$, and then a couple more jumps to know which case we are in. 

		Alternatively, we can avoid the precise calculations of these ordinals by working in the admissible set $L_{\w_1^B}[B]$, and performing a recursive construction in that structure; all steps are easily seen to be $\Delta_1$-definable in that structure. 		  
	\end{proof}
	
	Henceforth, we fix a rank-minimal set $R\in [B]^\w$ which is~$\Delta^1_1(B)$; for all~$n$ and~$E$ we let $\srk{n}{E} = \frk{R}{n}{E}$. It is the case that we will only consider $Z\in [R]^\w$, so $\srk{n}{E} = \frk{Z}{n}{E}$ for all such~$Z$, but this will not be important; this property of rank-minimal sets is only used in their construction. Rather, we will use the property of~$R$ stated in the definition of rank-minimality.

	\smallskip

	We no longer have a fixed presentation of the ordinals we are dealing with, so we cannot have a deficiency function as defined before. Instead, the following functions play a similar role: We say that a function $f\colon \w\to \w$ is a \emph{comparative deficiency function} (for~$R$) if:
	\begin{orderedlist}
		\item For all $n\in A$ and $E\in [B]^{<\w}$, if $\srk{n}{E}>0$ then for all $y \ge f(n,E)$ in~$R$ we have $\srk{n}{E\cup \{y\}} < \srk{n}{E}$. 
		\item For all $n,m\in A$ and all $E,F\in [B]^{<\w}$, if $\srk{n}{E} < \srk{m}{F}$ then for all $y \ge f(n,E;m,F)$ in~$R$  we have $\srk{n}{E} \le \srk{m}{F\cup \{y\}}$. 
	\end{orderedlist}
	The definition of rank-minimality ensures that comparative deficiency functions exist. Further, since the ranking function $\lambda^*$ is $\Delta^1_1(B\oplus R) = \Delta^1_1(B)$, there is a $\Delta^1_1(B)$ comparative deficiency function. Every function majorising a comparative deficiency function is also a comparative deficiency function. 

	\smallskip
	
	We are almost ready to prove the proposition.

	\begin{claim} \label{clm:last}
		There are sets $X,Y\in \Delta^1_1(B)$, $X\in [A]^\w$ and $Y\in [R]^\w$ such that:
		\begin{sublemma}
			\item $p_Y$ is a comparative deficiency function; and
			\item uniformly in $Z\in [Y]^\w$ we can compute a function $f_Z\colon \w\to \w$ such that:
			\begin{itemize}
				\item For all $n$, $f_Z(n)\in A$; 
				\item For all $n$, $n\in X$ if and only if $\srk{n}{\emptyset} \le \srk{f_Z(n)}{\emptyset}$. 
			\end{itemize}
		\end{sublemma}
	\end{claim}
	
	\begin{proof}
		There are two cases, depending on the patterns in the set $\left\{ \srk{n}{\emptyset} \,:\,  n\in A \right\}$. 

		\smallskip

		In the first case, for every $n\in A$, for almost all $m\in A$, $\srk{n}{\emptyset} < \srk{m}{\emptyset}$. In that case we let $X = A$ and $Y\in [R]^\w$ sufficiently sparse so that $p_Y$ is a comparative deficiency function, and so that for all $n\in A$, for all $m\ge p_Y(n)$ in~$A$, $\srk{n}{\emptyset} \le \srk{m}{\emptyset}$. We can choose $Y\in \Delta^1_1(B)$ since $R\in \Delta^1_1(B)$, there is a $\Delta^1_1(B)$ comparative deficiency function, and the relation $\srk{n}{\emptyset} \le \srk{m}{\emptyset}$ for $n,m\in A$ is $\Delta^1_1(B)$. Given $Z\in [Y]^\w$, the function $f_Z(n)$ is computed as follows: since $A \uie Y$, with oracle~$Z$ we enumerate~$A$ until we find some $m\in A$ greater than $p_Z(n)$, and set $f_Z(n)=m$. 

		\smallskip

		If the first case fails, there is some $n^*\in A$ such that for infinitely many~$n\in A$ we have $\srk{n}{\emptyset} \le \srk{n^*}{\emptyset}$. We let $X = \left\{ n\in A \,:\,  \srk{n}{\emptyset} \le \srk{n^*}{\emptyset} \right\}$. Again $X\in \Delta^1_1(B)$ as the ranking function is~$\Delta^1_1(B)$. We let the function~$f_Z$ be the constant function with value $n^*$. We let $Y\in [R]^\w$ be sufficiently sparse so that $p_Y$ is a comparative density function. 
	\end{proof}
	
	We can now prove the proposition. We claim that sets~$X$ and~$Y$ as given by \cref{clm:last} are as required: it remains to show that $X \uiT Y$. 

	\smallskip
	
	Let $Z\in [Y]^\w$, and let $n\in \w$. To decide whether $n\in X$, with oracle~$Z$, we perform the following procedure. Let $m = f_Z(n)$. We define two sequences $q_0 < q_1 < q_2 < \cdots$ and $r_0 < r_1 < r_2 < \cdots$ recursively; we let $E_i = \left\{ q_j \,:\,  j<i \right\}$ and $F_i = \left\{ r_j \,:\,  j<i \right\}$ (so $E_0 = F_0 = \emptyset$). 

	At step~$i\ge 0$, given~$E_i$ and~$F_i$, 
	\begin{itemize}
		\item First we choose $q_i\in Z$ greater than $p_Z(n,E_i)$ and $p_Z(m,F_i;n,E_i)$.
		\item Then we choose $r_i\in Z$ greater than $p_Z(m,F_i)$ and $p_Z(n,E_{i+1};m,F_i)$. 
	\end{itemize}
	The process halts at the least~$j$ such that $m\in \Phi(F_j)$. There will be such a~$j$ because otherwise, $F = \bigcup_j F_j$ is an infinite subset of~$B$ with $m\notin \Phi(F)$, however $m\in A$. Now we verify that 
	\[
		n\in X \Iff n\in \Phi(E_j). 
	\]

	First, suppose that $n\in X$; we show that $n\in \Phi(E_j)$. We prove by induction on~$i\le j$ that $\srk{n}{E_i} \le \srk{m}{F_i}$; this suffices, as we would get $\srk{n}{E_j} \le \srk{m}{F_j} = 0$. The case $i=0$ is given by the properties of~$X$ and $m = f_Z(n)$ as specified by \cref{clm:last}. For the inductive step, let $i<j$ and suppose that $\srk{n}{E_i}\le \srk{m}{F_j}$. If $\srk{n}{E_i}= 0$ then certainly $\srk{n}{E_{i+1}} = 0 \le \srk{m}{F_{i+1}}$ (and in fact, $n\in \Phi(E_i)$ so $n\in \Phi(E_j)$ which is what we really want to prove). Suppose that $\srk{n}{E_i}>0$. Since $q_i > p_Y(n,E_i)$ is an element of~$R$, we have $\srk{n}{E_{i+1}} < \srk{n}{E_i}$ and so $\srk{n}{E_{i+1}} < \srk{m}{F_j}$. Since $r_i > p_Y(n,E_{i+1};m,F_i)$, we have $\srk{n}{E_{i+1}}\le \srk{m}{F_{i+1}}$ as required. 

	Next, suppose that $n\notin X$; we show that $n\notin \Phi(E_j)$. If $n\notin A$ this is clear, so we may assume that $n\in A$. We show that $\srk{n}{E_j}>0$. By induction on $i\le j$ we show that $\srk{n}{E_i} > \srk{m}{F_i}$; then we will have $\srk{n}{E_j} > \srk{m}{F_j} \ge 0$. Again for $i=0$ this is by the choice of~$X$ and $m = f_Z(n)$. Suppose that $i<j$ and $\srk{n}{E_i} > \srk{m}{F_i}$. Since $q_i > p_Y(m,F_i;n,E_i)$ and is an element of~$R$, we have $\srk{n}{E_{i+1}} \ge \srk{m}{F_i}$. Since $i<j$, the minimality of~$j$ implies $\srk{m}{F_i}>0$, and so since $r_i> p_Y(m,F_i)$ we have $\srk{m}{F_{i+1}}< \srk{m}{F_i} \le \srk{n}{E_{i+1}}$, as required. 

	\medskip
	
	This completes the proof of \cref{prop:improving:main}. We remark that the computation process just described can be simplified in the first case of \cref{clm:last}, because that case gives $X = A$; we only need to require that $q_i > p_Y(n,E_i)$ and $r_i > p_Y(n,E_{i+1};m,F_i)$. In that case we get $A \uiT Y$, that is, \cref{prop:1.6_strengthened} (and so \cref{thm:improving_enumerability_one_sided}) holds for~$A$ and~$B$ without the assumption that~$A$ is introenumerable. That assumption is only used when we are at the second case of~\cref{clm:last}, when $X\ne A$. \Cref{thm:failure_of_improving_enumerability}\ref{item:failure_right} shows that this case does happen. 
\end{proof}


\bibliography{References}

\newcommand{\etalchar}[1]{$^{#1}$}
\begin{thebibliography}{HJKH{\etalchar{+}}08}

\bibitem[Ast15]{Astor15}
Eric~P. Astor.
\newblock {\em Asymptotic density and effective negligibility}.
\newblock ProQuest LLC, Ann Arbor, MI, 2015.
\newblock Thesis (Ph.D.)--The University of Chicago.

\bibitem[DM58]{DekkerMyhill58}
J.~C.~E. Dekker and J.~Myhill.
\newblock Retraceable sets.
\newblock {\em Canadian J. Math.}, 10:357--373, 1958.

\bibitem[FR59]{FriedbergRogers}
Richard~M. Friedberg and Hartley Rogers, Jr.
\newblock Reducibility and completeness for sets of integers.
\newblock {\em Z. Math. Logik Grundlagen Math.}, 5:117--125, 1959.

\bibitem[FRSM19]{FokinaRosseggerSanMauro}
Ekaterina Fokina, Dino Rossegger, and Luca San~Mauro.
\newblock Measuring the complexity of reductions between equivalence relations.
\newblock {\em Computability}, 8(3-4):265--280, 2019.

\bibitem[Gan60]{Gandy:basis}
R.~O. Gandy.
\newblock On a problem of {K}leene's.
\newblock {\em Bull. Amer. Math. Soc.}, 66:501--502, 1960.

\bibitem[GP73]{GalvinPrikry}
Fred Galvin and Karel Prikry.
\newblock Borel sets and {R}amsey's theorem.
\newblock {\em J. Symbolic Logic}, 38:193--198, 1973.

\bibitem[Hir20]{Hirschfeldt20}
Denis~R. Hirschfeldt.
\newblock A minimal pair in the generic degrees.
\newblock {\em J. Symb. Log.}, 85(1):531--537, 2020.

\bibitem[HJKH{\etalchar{+}}08]{Hirschfeldt2008strength}
Denis~R. Hirschfeldt, Carl~G. Jockusch, Bj{\o}rn Kjos-Hanssen, Steffen Lempp,
  and Theodore~A. Slaman.
\newblock {The strength of some combinatorial principles related to {R}amsey's
  theorem for pairs}.
\newblock {\em Computational Prospects of Infinity, Part II: Presented Talks,
  World Scientific Press, Singapore}, pages 143--161, 2008.

\bibitem[HJMS16]{HJMS}
Denis~R. Hirschfeldt, Carl~G. Jockusch, Jr., Timothy~H. McNicholl, and Paul~E.
  Schupp.
\newblock Asymptotic density and the coarse computability bound.
\newblock {\em Computability}, 5(1):13--27, 2016.

\bibitem[Igu13]{Igusa13}
Gregory Igusa.
\newblock {\em Generic {R}eduction, and {W}ork with {P}artial {C}omputations
  and {P}artial {O}racles}.
\newblock ProQuest LLC, Ann Arbor, MI, 2013.
\newblock Thesis (Ph.D.)--University of California, Berkeley.

\bibitem[Joc68]{Jockusch68:Introreducible}
Carl~G. Jockusch, Jr.
\newblock Uniformly introreducible sets.
\newblock {\em J. Symbolic Logic}, 33:521--536, 1968.

\bibitem[JS12]{MR2901074}
Carl~G. Jockusch, Jr. and Paul~E. Schupp.
\newblock Generic computability, {T}uring degrees, and asymptotic density.
\newblock {\em J. Lond. Math. Soc. (2)}, 85(2):472--490, 2012.

\bibitem[JS17]{JockuschSchupp17}
Carl~G. Jockusch, Jr. and Paul~E. Schupp.
\newblock Asymptotic density and the theory of computability: a partial survey.
\newblock In {\em Computability and complexity}, volume 10010 of {\em Lecture
  Notes in Comput. Sci.}, pages 501--520. Springer, Cham, 2017.

\bibitem[Liu12]{Liu2012RT22}
Lu~Liu.
\newblock {RT$^2_2$ does not imply WKL$_0$}.
\newblock {\em Journal of Symbolic Logic}, 77(2):609--620, 2012.

\bibitem[Liu15]{Liu2015Cone}
Lu~Liu.
\newblock Cone avoiding closed sets.
\newblock {\em Transactions of the American Mathematical Society},
  367(3):1609--1630, 2015.

\bibitem[MP19]{monin2019srt22}
Benoit Monin and Ludovic Patey.
\newblock Srt22 does not imply rt22 in omega-models.
\newblock {\em arXiv preprint arXiv:1905.08427}, 2019.

\bibitem[Pat17]{Patey2017Controlling}
Ludovic Patey.
\newblock Controlling iterated jumps of solutions to combinatorial problems.
\newblock {\em Computability}, 6(1):47--78, 2017.

\bibitem[Soa69]{Soare:higher}
Robert~I. Soare.
\newblock Sets with no subset of higher degree.
\newblock {\em J. Symbolic Logic}, 34:53--56, 1969.

\bibitem[Soa87]{Soare1987Recursively}
Robert~I. Soare.
\newblock {\em Recursively enumerable sets and degrees: A study of computable
  functions and computably generated sets}.
\newblock Springer Science \& Business Media, 1987.

\bibitem[Sol78]{Solovay78}
Robert~M. Solovay.
\newblock Hyperarithmetically encodable sets.
\newblock {\em Trans. Amer. Math. Soc.}, 239:99--122, 1978.

\bibitem[SS95]{SeetapunSlaman95}
David Seetapun and Theodore~A. Slaman.
\newblock On the strength of {R}amsey's theorem.
\newblock volume~36, pages 570--582. 1995.
\newblock Special Issue: Models of arithmetic.

\bibitem[SW86]{SlamanWoodin}
Theodore~A. Slaman and W.~Hugh Woodin.
\newblock Definability in the {T}uring degrees.
\newblock {\em Illinois J. Math.}, 30(2):320--334, 1986.

\end{thebibliography}
\bibliographystyle{alpha}

\end{document}